\documentclass[11pt]{amsart}
\usepackage{amsmath,amsfonts,amssymb,amscd,amsthm,bm}
\usepackage{mathtools}
\usepackage{mathrsfs}
\usepackage{graphicx}
\usepackage{booktabs}
\usepackage{tabularx}
\usepackage[square, comma, sort&compress, numbers]{natbib}
\usepackage{tikz}
\usepackage{float}
\usepackage{hyperref}
\usetikzlibrary{positioning}

\allowdisplaybreaks
\newcolumntype{Y}{>{\raggedleft\arraybackslash}X}

\newcommand{\bz}{\mathbb{Z}}

\def\id{\text{\rm Id\,}}

\def\dim{\text{\rm dim\,}}

\newcommand{\eqdeg}[1]{#1\mbox{\rm -}\deg}

\newcommand{\amal}[5]{#1\prescript{#4}{}\times_{#3}^{#5}#2}
\newtheorem{theorem}{Theorem}[section]

\newtheorem{lemma}{Lemma}[section]
\newtheorem{corollary}{Corollary}[section]
\newtheorem{condition}{Condition}[section]

\newtheorem{remark}{Remark}[section]
\begin{document}
\title{Global Bifurcation and Symmetry of Periodic Inventory Oscillations in Ring Supply-Chain Networks}

\author{Shi Yu}
\address{Math Department, Odessa College, 201 W University Blvd, Odessa, TX 79764, USA}
\email{syu@odessa.edu}

\begin{abstract}
This paper studies periodic inventory oscillations in delayed supply-chain networks of identical warehouses arranged in a ring.  After reducing the inventory--replenishment system to a second-order delay equation, we formulate the periodic problem as an equivariant compact perturbation of the identity, with symmetries generated by the ring structure, the odd nonlinearity, and time translations. Using equivariant degree theory, we characterize critical parameter pairs, compute the associated crossing numbers, and construct a local bifurcation invariant whose nonvanishing guarantees nonconstant periodic solutions and provides information about their spatial and spatio-temporal symmetry types. By restricting the two-parameter problem to an admissible transversal curve, we further establish a global continuation alternative for the bifurcating branches. Examples with \(D_7\)- and \(D_8\)-symmetry illustrate the computation of critical parameters, bifurcation invariants, and admissible orbit types. Finally, Floquet theory is applied to small-amplitude periodic branches, yielding a leading-order stability criterion for nonresonant perturbation modes. The results provide a unified framework for describing the existence, symmetry, global continuation, and leading-order stability of periodic inventory oscillations in delayed ring supply-chain networks.
\end{abstract}

\subjclass[2020]{34K18, 47J15, 37G15, 34K13}

\keywords{Supply-chain networks, delay differential equations, periodic solutions, equivariant degree, global bifurcation, Floquet multipliers, spatio-temporal symmetry}

\maketitle

\section{Introduction}
In supply-chain networks, replenishment decisions are rarely based on the instantaneous state of the system. Inventory records, replenishment orders, and information received from neighboring locations typically reflect conditions from an earlier time. Delays in demand information, replenishment, and inter-location communication may therefore generate persistent oscillations in inventory and production levels; see, for example, \cite{Forrester1961,Hu1,Hu2,Sterman1989}. This mechanism is closely related to the bullwhip effect, in which relatively small variations in customer demand are amplified as they propagate through the supply chain \cite{Forrester1961,Hu1,LeePadmanabhanWhang1997}. Related symmetry and bifurcation phenomena in structured dynamical systems can be treated within the general framework of symmetric bifurcation theory; see \cite{GolSchSt}.

Motivated by this phenomenon, we study a delayed ring network of identical warehouses with local inventory feedback, delayed nearest-neighbor coordination, and nonlinear saturation. The ring architecture gives rise to rotational and reflectional symmetries, while the odd nonlinearity and time translations introduce additional sign and temporal symmetries. These structural properties make it possible to analyze periodic inventory oscillations within an equivariant bifurcation framework.

The main novelty of this paper is a symmetry-sensitive local and global bifurcation analysis for periodic solutions of the resulting delayed supply-chain model. In contrast to approaches that detect oscillations only through characteristic roots or ordinary bifurcation indices, the equivariant degree also distinguishes the spatial and spatio-temporal symmetry types of the bifurcating solutions. In particular, the method identifies possible homogeneous, alternating, and wave-like oscillatory patterns and provides a global continuation alternative for the corresponding solution branches.

More precisely, we first reduce the inventory--replenishment system to a second-order delay equation and reformulate the periodic problem as an equivariant compact perturbation of the identity. We then characterize the critical parameter values through purely imaginary characteristic roots and compute the associated crossing numbers on the spatio-temporal isotypical components. These contributions are combined in a local bifurcation invariant. A nonzero invariant yields the existence of nonconstant periodic solutions together with information about their possible orbit types. Finally, by restricting the two-parameter problem to an admissible transversal curve, we establish a global continuation result showing that a bifurcating branch cannot remain confined to a bounded neighborhood of the critical point.

The analysis therefore goes beyond the mere detection of periodic inventory oscillations by relating the delay parameters, the spatial structure of the network, and the symmetry of the resulting oscillatory patterns. Equivariant bifurcation and degree methods provide a natural framework for such problems because they retain information about isotropy and orbit types that is not captured by ordinary topological degree. The general theory of equivariant degree and its computational framework can be found in \cite{AED,SURVEY,diss}, while related applications to symmetric periodic and subharmonic solutions, reversible systems, ring configurations, delayed networks, and other equivariant differential equations include \cite{SAC,SAC2,GI4,GI5,Guo,2nd,SY2,LiuGarciaKrawcewicz2023,Garcia, Ghanem2025,Yu3,Yu2026Random,Crane2026}. Global continuation and Hopf-type phenomena for delay differential equations are also closely related to the approaches developed in \cite{HuWuJDE,HuWuJDDE,HuWuZou}. These works provide the broader theoretical context for the symmetry-sensitive local and global bifurcation analysis developed here.

\begin{figure}[htbp]
\centering
\begin{tikzpicture}[scale=1.15, every node/.style={font=\small}]

\def\n{8}
\pgfmathtruncatemacro{\nlast}{\n-1}
\def\radius{2.4}

\foreach \k in {1,...,\n}
{
    \node[circle, draw, thick, minimum size=9mm]
    (I\k) at ({90-360*(\k-1)/\n}:\radius) {\(I_{\k}\)};
}

\foreach \k in {1,...,\nlast}
{
    \pgfmathtruncatemacro{\j}{\k+1}
    \draw[<->, thick] (I\k) -- (I\j);
}
\draw[<->, thick] (I\n) -- (I1);

\draw[dashed, thick, bend left=18, ->] (I1) to node[above right] {\(\tau_c\)} (I2);
\draw[dashed, thick, bend left=18, ->] (I2) to node[right] {\(\tau_c\)} (I1);

\node[align=center] at (0,0) {Delayed nearest-neighbor\\coordination};

\pgfmathtruncatemacro{\bottomidx}{floor(\n/2)+1}
\node[draw, rounded corners, align=center, below=1.0cm of I\bottomidx, font=\small] (delay)
{Inventory information delay\\ \(I_k(t-\tau_d)\)};

\draw[->, thick] (delay) -- (I\bottomidx);

\end{tikzpicture}
\caption{A symmetric ring supply-chain network of \(n\) warehouses with delayed local information and delayed nearest-neighbor coordination. The closed-loop structure gives the model the natural spatial symmetry \(D_n\).}
\label{fig:ring-supply-chain}
\end{figure}
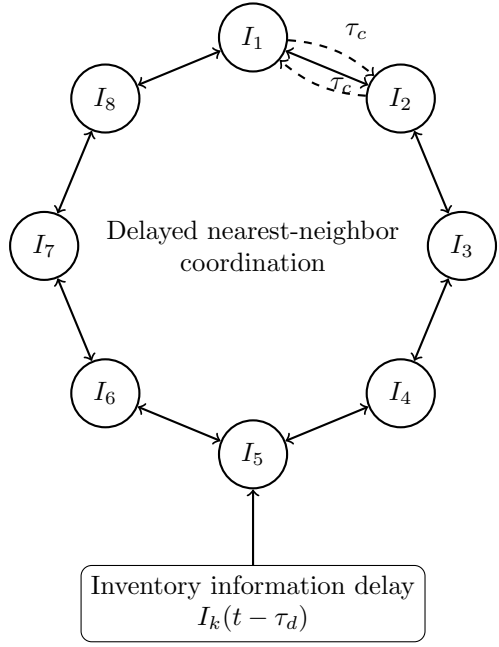
The delayed ring supply-chain structure is illustrated in
Figure~\ref{fig:ring-supply-chain}. We consider a ring network of
\(n\) warehouses governed by
\begin{equation}\label{eq:modified-system}
\begin{cases}
\dot I_k(t)
=
P_k(t-\tau_p)-D_k(t)
-\delta\bigl(I_k(t)-S_k\bigr),\\[1mm]
\begin{aligned}
\dot P_k(t)
={}&-a\bigl(P_k(t)-P_0\bigr)
+b\bigl(S_k-I_k(t-\tau_d)\bigr)\\
&+c\bigl(I_{k+1}(t-\tau_c)-2I_k(t)
+I_{k-1}(t-\tau_c)\bigr)\\
&-\beta\bigl(I_k(t)-S_k\bigr)^3,
\end{aligned}
\end{cases}
\qquad k=1,\ldots,n,
\end{equation}
where the indices are taken modulo \(n\), so that
\[
I_0=I_n,
\qquad
I_{n+1}=I_1.
\]
The state variables, reference quantities, and model parameters are listed in
Table~\ref{tab:model-quantities}.

\begin{table}[H]
\centering
\caption{Variables and parameters in the supply-chain model.}
\label{tab:model-quantities}
\begin{tabular}{>{\centering\arraybackslash}p{2.2cm}p{10.0cm}}
\toprule
Symbol & Meaning \\
\midrule
\(I_k(t)\) & Inventory level at warehouse \(k\). \\
\(P_k(t)\) & Replenishment rate at warehouse \(k\). \\
\(D_k(t)\) & Demand rate at warehouse \(k\). \\
\(S_k\) & Desired inventory level at warehouse \(k\). \\
\(P_0\) & Baseline replenishment rate. \\
\(a>0\) & Replenishment relaxation rate. \\
\(\delta>0\) & Inventory correction rate relative to \(S_k\). \\
\(b>0\) & Delayed inventory-feedback strength. \\
\(c\geq0\) & Nearest-neighbor coordination strength. \\
\(\beta>0\) & Cubic saturation coefficient. \\
\(\tau_p\geq0\) & Replenishment transport delay. \\
\(\tau_d\geq0\) & Local inventory-information delay. \\
\(\tau_c\geq0\) & Neighbor-coordination delay. \\
\bottomrule
\end{tabular}
\end{table}

We assume that the warehouses are identical:
\[
S_k=S,
\qquad
D_k(t)=D_0,
\qquad
\tau_d=\tau_0,
\qquad
k=1,\ldots,n.
\]
We also assume
\[
\tau_p=0
\qquad\text{and}\qquad
P_0=D_0.
\]
Under these assumptions, the spatially homogeneous state
\[
I_k(t)=S,
\qquad
P_k(t)=D_0,
\qquad
k=1,\ldots,n,
\]
is an equilibrium of \eqref{eq:modified-system}.

Introduce the perturbation variables
\[
x_k(t)=I_k(t)-S,
\qquad
y_k(t)=P_k(t)-D_0.
\]
and the system becomes
\begin{equation}\label{eq:supply-chain-two-param}
\begin{cases}
\dot x_k(t)=y_k(t)-\delta x_k(t),\\[1mm]
\begin{aligned}
\dot y_k(t)
={}&-a y_k(t)-b x_k(t-\tau_0)\\
&+c\bigl(x_{k+1}(t-\tau_c)-2x_k(t)
+x_{k-1}(t-\tau_c)\bigr)
-\beta x_k^3(t).
\end{aligned}
\end{cases}
\end{equation}
Differentiating the first equation in \eqref{eq:supply-chain-two-param} and using
\[
y_k(t)=\dot x_k(t)+\delta x_k(t),
\]
we obtain
\[
\ddot x_k(t)+(a+\delta)\dot x_k(t)+a\delta x_k(t)+b x_k(t-\tau_0)
-c\bigl(x_{k+1}(t-\tau_c)-2x_k(t)+x_{k-1}(t-\tau_c)\bigr)
+\beta x_k^3(t)=0.
\]
Setting
\[
\eta=a+\delta,\qquad \alpha=a\delta,
\]
gives 
\begin{equation}\label{eq:reduced-supply-chain-system}
\begin{aligned}
\ddot x_k(t)
&+\eta\dot x_k(t)
+\alpha x_k(t)
+b x_k(t-\tau_0)\\
&-c\bigl(x_{k+1}(t-\tau_c)-2x_k(t)
+x_{k-1}(t-\tau_c)\bigr)
+\beta x_k^3(t)=0,
\qquad k=1,\ldots,n.
\end{aligned}
\end{equation}
For the bifurcation analysis, define the parameter domain
\[
\mathfrak P:=(0,\infty)\times[0,\infty), 
\]
then
\[
(b,\tau_0)\in\mathfrak P
\]
as the two bifurcation parameters and keep
\(a, \delta, c, \beta, \tau_c \)
fixed. 

Since our goal is to study periodic inventory oscillations, we impose the time-periodicity condition
\[
\quad x_k(t+2\pi)=x_k(t),\qquad k=1,\ldots,n.
\]

The periodic spatial boundary condition means that the supply-chain network is arranged as a closed loop. Hence \eqref{eq:reduced-supply-chain-system} is invariant under rotations and reflections of the warehouse indices, and the spatial symmetry group is the dihedral group \(D_n\). Since the reduced equation contains only odd powers of \(x_k\), it is also invariant under the sign change \(x\mapsto -x\), giving an additional \(\mathbb Z_2\)-symmetry. With damping and one-sided delays, the time symmetry available to periodic solutions is the time-translation group \(S^1\). Therefore, the relevant symmetry group is
\[
G=D_n\times \mathbb Z_2\times S^1.
\]

The main contribution of this paper is the integration of a delayed supply-chain ring model with an equivariant degree framework that establishes the existence of periodic oscillations and describes their spatial symmetry and global continuation.

\section{Function Space Setting and Symmetry Group Structure}
We study small periodic deviations from the homogeneous inventory equilibrium of the supply-chain ring. We write the unknown as
\[
x(t)=(x_1(t),\ldots,x_n(t)).
\]
Since we are interested in \(2\pi\)-periodic oscillations, we work in the space
\[
\mathcal X:=H^2_{\mathrm{per}}([0,2\pi],\mathbb R^n).
\]
We also set
\[
\mathcal Y:=L^2([0,2\pi],\mathbb R^n).
\]
The space \(\mathcal X\) is equipped with the inner product
\[
\langle x,z\rangle_{\mathcal X}=\int_0^{2\pi}\left(x''(t)\cdot z''(t)+x'(t)\cdot z'(t)+x(t)\cdot z(t)\right)\,dt.
\]
The induced norm is
\[
\|x\|_{\mathcal X}=\left(\int_0^{2\pi}\left(|x''(t)|^2+|x'(t)|^2+|x(t)|^2\right)\,dt\right)^{1/2}.
\]
For
\[
g=(\gamma,\iota,\theta)\in G,
\]
with
\[
\gamma\in D_n,
\qquad
\iota\in\mathbb Z_2,
\qquad
\theta\in S^1,
\]
we define the action on \(\mathcal X\) by
\[
(gx)(t)
=\gamma\,\iota\,x(t+\theta).
\]
The same action is used on \(\mathcal Y\). With this convention,
\(\mathcal X\) and \(\mathcal Y\) are \(G\)-representations.

The temporal Fourier decomposition of \(\mathcal X\) is
\[
\mathcal X
=
\overline{\bigoplus_{m=0}^{\infty}E_m},
\]
where
\[
E_0
=
\{a:a\in\mathbb R^n\},
\]
and, for \(m\geq1\),
\[
E_m
=
\left\{
a\cos(mt)+b\sin(mt):
a,b\in\mathbb R^n
\right\}.
\]
Thus \(E_m\) is the real temporal Fourier subspace associated with frequency
\(m\).

To describe the spatial decomposition, let
\[
\Lambda_n
=
\begin{cases}
\left\{0,1,\ldots,\dfrac{n-1}{2}\right\},
& n \text{ odd},\\[2mm]
\left\{0,1,\ldots,\dfrac{n}{2}\right\},
& n \text{ even}.
\end{cases}
\]
We consider the \(D_n\times\mathbb Z_2\)-isotypical decomposition
\[
\mathbb R^n
=
\bigoplus_{\ell\in\Lambda_n}V_\ell.
\]
For \(1\leq\ell<\frac n2\), 
the real spatial component \(V_\ell\) is generated by the paired complex
Fourier modes \(\ell\) and \(n-\ell\). The factor \(\mathbb Z_2\) acts on
each \(V_\ell\) through the sign representation.

Combining the temporal and spatial decompositions, we obtain
\[
E_m
=
\bigoplus_{\ell\in\Lambda_n}\mathcal V_{\ell,m},
\qquad m\geq0,
\]
where
\[
\mathcal V_{\ell,0}=V_\ell,
\]
and, for \(m\geq1\),
\[
\mathcal V_{\ell,m}
=
\mathcal W_m\otimes_{\mathbb R}V_\ell.
\]
Here \(\mathcal W_m\) denotes the two-dimensional real irreducible
representation of \(S^1\) associated with temporal frequency \(m\), which is  

\[
\operatorname{span}\{\cos(mt),\sin(mt)\}.
\]
Consequently,
\[
\mathcal X
=
\overline{
\bigoplus_{m=0}^{\infty}
\bigoplus_{\ell\in\Lambda_n}
\mathcal V_{\ell,m}
}.
\]
Define the operator \(\mathscr F_{b,\tau_0}:\mathcal X\to\mathcal Y\) by rewriting \eqref{eq:reduced-supply-chain-system} in the form
\[
x''(t)=\mathscr F_{b,\tau_0}(x)(t),
\]
where the \(k\)-th component of \(\mathscr F_{b,\tau_0}(x)\) is given by
\[
\mathscr F_{b,\tau_0,k}(x)(t)
=
-\eta x_k'(t)
-\alpha x_k(t)
-bx_k(t-\tau_0)
+c\bigl(x_{k+1}(t-\tau_c)-2x_k(t)+x_{k-1}(t-\tau_c)\bigr)
-\beta x_k^3(t).
\]
Since \(\mathcal X\)
is compactly embedded into \(C^1_{\mathrm{per}}([0,2\pi];\mathbb R^n),
\) the mapping \(x\longmapsto x^3\)
is well defined from \(\mathcal X\) into \(\mathcal Y\) and sends bounded subsets of \(\mathcal X\) into bounded subsets of \(\mathcal Y\). Moreover, the translation operators
\[
x(t)\longmapsto x(t-\tau_0),
\qquad
x(t)\longmapsto x(t-\tau_c),
\]
are bounded on both \(\mathcal X\) and \(\mathcal Y\). Therefore, for every \((b,\tau_0)\in\mathfrak P,
\) the operator \(\mathscr F_{b,\tau_0}\) is a well-defined mapping from \(\mathcal X\) into \(\mathcal Y\).

To reformulate the problem as a compact perturbation of the identity, introduce the linear operator
\[
\mathcal A:\mathcal X\to\mathcal Y,
\qquad
\mathcal A x=x''-x.
\]

\begin{lemma}\label{lem:A-isomorphism}
The operator \(\mathcal A:\mathcal X\to\mathcal Y\) is an isomorphism.
\end{lemma}

\begin{proof}
Let
\[
x(t)=a_0+\sum_{m=1}^{\infty}\big(a_m\cos(mt)+b_m\sin(mt)\big)
\]
be the Fourier expansion of \(x\). Then
\[
\mathcal A x=-a_0-\sum_{m=1}^{\infty}(m^2+1)\big(a_m\cos(mt)+b_m\sin(mt)\big).
\]
Since \(m^2+1\neq 0\) for every \(m\geq 0\), no Fourier mode lies in the kernel of \(\mathcal A\). Moreover, division by \(m^2+1\) gives a bounded inverse from \(\mathcal Y\) into \(\mathcal X\). Hence \(\mathcal A\) is an isomorphism.
\end{proof}

Let
\[
\mathfrak i:\mathcal X\hookrightarrow C_{\mathrm{per}}^1(S^1;\mathbb R^n)
\]
denote the natural embedding, and define the associated Nemytskii operator
\[
N_{\mathscr F_{b,\tau_0}}:
C_{\mathrm{per}}^1(S^1;\mathbb R^n)
\longrightarrow \mathcal Y
\]
by
\[
N_{\mathscr F_{b,\tau_0}}(v)(t)=\mathscr F_{b,\tau_0}(v)(t).
\]

Then the equation
\[
x''=\mathscr F_{b,\tau_0}(x)
\]
is equivalent to
\[
\mathcal A x
=
N_{\mathscr F_{b,\tau_0}}(\mathfrak ix)-x.
\]

Using the inverse of \(\mathcal A\), define
\[
\Psi:\mathfrak P\times\mathcal X\to\mathcal X,\qquad \Psi(b,\tau_0,x)=\Psi_{b,\tau_0}(x),
\]
where
\[
\Psi_{b,\tau_0}(x):=x-\mathcal A^{-1}\big(N_{\mathscr F_{b,\tau_0}}(\mathfrak ix)-x\big).
\]

Therefore, the zeros of \(\Psi_{b,\tau_0}\) are exactly the \(2\pi\)-periodic solutions of \eqref{eq:reduced-supply-chain-system} for the given parameter pair \((b,\tau_0)\).

\begin{theorem}\label{thm:Gamma-equivariance}
For every \((b,\tau_0)\in\mathfrak P\), the map
\[
\Psi_{b,\tau_0}:\mathcal X\to\mathcal X
\]
is \(G\)-equivariant.
Moreover, \(\Psi_{b,\tau_0}\) is a compact perturbation of the identity and depends continuously on \((b,\tau_0)\), uniformly on bounded subsets of \(\mathcal X\).
\end{theorem}

\begin{proof}
Since \(\mathcal A\) has constant coefficients and acts only on the time variable, it commutes with time translations, with the coordinate permutations induced by \(D_n\), and with the sign action of \(\mathbb Z_2\). Consequently, both \(\mathcal A\) and \(\mathcal A^{-1}\) are \(G\)-equivariant. The embedding \(\mathfrak i\) is also \(G\)-equivariant, since
\[
\mathfrak i(gx)=g(\mathfrak ix),\qquad g\in G,\quad x\in\mathcal X.
\]

Since the components of \(\mathscr F_{b,\tau_0}\) have identical coefficients and the nearest-neighbor coupling is preserved by rotations and reflections of the ring, \(\mathscr F_{b,\tau_0}\) is \(D_n\)-equivariant. The oddness of the nonlinear terms implies equivariance under the sign action of \(\mathbb Z_2\). Furthermore, because the system is autonomous, differentiation and the delay operators commute with time translations. Consequently,
\[
\mathscr F_{b,\tau_0}(gu)=g\mathscr F_{b,\tau_0}(u),\qquad g\in G,\quad u\in C_{\mathrm{per}}^1([0,2\pi];\mathbb R^n).
\]

Therefore, using the equivariance of \(\mathfrak i\), \(\mathscr F_{b,\tau_0}\), and \(\mathcal A^{-1}\), we obtain
\[
\begin{aligned}
\Psi_{b,\tau_0}(gx)
&=gx-\mathcal A^{-1}\bigl(N_{\mathscr F_{b,\tau_0}}(\mathfrak i(gx))-gx\bigr)\\
&=gx-\mathcal A^{-1}g\bigl(N_{\mathscr F_{b,\tau_0}}(\mathfrak ix)-x\bigr)\\
&=g\Psi_{b,\tau_0}(x).
\end{aligned}
\]
Thus, \(\Psi_{b,\tau_0}\) is \(G\)-equivariant.

We next prove compactness. The embedding \(\mathfrak i\)
is compact. Hence, if \(B\subset\mathcal X\) is bounded, then \(\mathfrak i(B)\) is relatively compact in \(
C_{\mathrm{per}}^1([0,2\pi];\mathbb R^n). \)
The delay operators, differentiation, and the Nemytskii operators appearing in \(\mathscr F_{b,\tau_0}\) are continuous from \(C_{\mathrm{per}}^1([0,2\pi];\mathbb R^n)\)
into \(\mathcal Y\). It follows that
\[
x\longmapsto N_{\mathscr F_{b,\tau_0}}(\mathfrak ix)-x
\]
maps bounded subsets of \(\mathcal X\) into relatively compact subsets of \(\mathcal Y\). Since \(\mathcal A^{-1}\)
is bounded, the map
\[
\mathcal A^{-1}\bigl(N_{\mathscr F_{b,\tau_0}}(\mathfrak ix)-x\bigr)
\]
is compact. Therefore, \(\Psi_{b,\tau_0}\)
is a compact perturbation of the identity.

Finally, the translation map
\[
(\tau,x)\longmapsto x(\,\cdot-\tau)
\]
is continuous from bounded subsets of
\[
[0,\infty)\times C_{\mathrm{per}}^1([0,2\pi];\mathbb R^n)
\]
into \(C_{\mathrm{per}}^1([0,2\pi];\mathbb R^n)\). Therefore,
\[
(b,\tau_0,x)\longmapsto
\mathcal A^{-1}\bigl(
N_{\mathscr F_{b,\tau_0}}(\mathfrak ix)-x
\bigr)
\]
depends continuously on \((b,\tau_0)\), uniformly for \(x\) in bounded subsets
of \(\mathcal X\).
\end{proof}

Consequently, if \(\mathcal U\subset\mathcal X\) is a bounded open \(G\)-invariant set and
\[
\Psi_{b,\tau_0}(x)\neq 0\qquad \text{for all }x\in\partial\mathcal U,
\]
then the corresponding \(G\)-equivariant degree
\[
\eqdeg{G}\bigl(\Psi_{b,\tau_0},\mathcal U\bigr)
\]
is well defined. 
For a compact Lie group \(G\), let \(\Phi(G)\) denote the set of
conjugacy classes of closed subgroups of \(G\). For a closed subgroup
\(H\leq G\), define its normalizer by
\[
N(H):=\{g\in G:gHg^{-1}=H\},
\]
and its Weyl group by
\[
W(H):=N(H)/H.
\]

For \(k\geq0\), set
\[
\Phi_k(G)
:=
\left\{
(H)\in\Phi(G):
\dim W(H)=k
\right\}.
\]

We define
\[
A_0(G):=\mathbb Z[\Phi_0(G)],
\]
the free abelian group generated by the conjugacy classes
\((H)\) for which
\[
\dim W(H)=0.
\]
Since in this case \(W(H)\) is finite, \(A_0(G)\) is the usual
Burnside ring of \(G\).

Similarly, we define
\[
A_1(G):=\mathbb Z[\Phi_1(G)],
\]
the free abelian group generated by the conjugacy classes
\((H)\) satisfying
\[
\dim W(H)=1.
\]
The group \(A_1(G)\) is used to describe the one-parameter
equivariant degree associated with the nonstationary temporal modes.

For the stationary component, the ordinary equivariant degree takes values in the Burnside ring \(A_0(G)\), whereas the nonstationary temporal components are described by the one-parameter equivariant degree with values in \(A_1(G)\); see \cite{AED,diss}. By Theorem~\ref{thm:Gamma-equivariance} and the homotopy invariance of the corresponding equivariant degrees, these degree elements remain unchanged under continuous variations of
\((b,\tau_0)\) as long as
\[
\Psi_{b,\tau_0}(x)\neq0
\qquad
\text{for all }x\in\partial\mathcal U.
\]
Consequently, a change in the equivariant degree can occur only when the parameter pair \((b,\tau_0)\) passes through a critical set at which the admissibility condition fails.
\section{Local Bifurcation from the Equilibrium Branch}
For a sufficiently small \(G\)-invariant neighborhood \(\mathcal U\) of the origin, the nonlinear map has the same local equivariant degree as its linearization whenever the latter is invertible. Thus, if
\[
D_x\Psi_{b,\tau_0}(0):\mathcal X\to\mathcal X
\]
is invertible, then
\[
\eqdeg{G}\big(\Psi_{b,\tau_0},\mathcal U\big)
=
\eqdeg{G}\big(D_x\Psi_{b,\tau_0}(0),\mathcal U\big).
\]
This follows from the admissible homotopy between \(\Psi_{b,\tau_0}\) and its linear part near the isolated trivial solution; see \cite{AED,diss}. Hence possible bifurcation points are found from the loss of invertibility of the linearization.

From \eqref{eq:supply-chain-two-param}, the linearized equation is
\[
x_k''(t)+\eta x_k'(t)+\alpha x_k(t)+b x_k(t-\tau_0)-c\big(x_{k+1}(t-\tau_c)-2x_k(t)+x_{k-1}(t-\tau_c)\big)=0.
\]
Let
\[
\vartheta_\ell=\frac{2\pi\ell}{n},\qquad \ell=0,1,\ldots,n-1.
\]
For the \(\ell\)-th spatial Fourier mode,
\[
x_k(t)=y(t)e^{i\vartheta_\ell k},
\]
the coupling term satisfies
\[
x_{k+1}(t-\tau_c)+x_{k-1}(t-\tau_c)
=
2\cos(\vartheta_\ell)y(t-\tau_c)e^{i\vartheta_\ell k}.
\]
Thus the linearized equation reduces to
\[
y''(t)+\eta y'(t)+(\alpha+2c)y(t)+b y(t-\tau_0)-2c\cos(\vartheta_\ell)y(t-\tau_c)=0.
\]

Substituting
\[
y(t)=e^{\nu t}
\]
gives
\[
\Delta_\ell(\nu,b,\tau_0)=\nu^2+\eta\nu+\alpha+2c+b e^{-\nu\tau_0}
-2c\cos(\vartheta_\ell)e^{-\nu\tau_c}.
\]
A \(2\pi\)-periodic critical mode corresponds to
\[
\nu=im,\qquad m\in\mathbb Z.
\]

We obtain
\[
\Delta_\ell(im,b,\tau_0)=P_{\ell,m}(b,\tau_0)+iQ_{\ell,m}(b,\tau_0),
\]
where
\begin{equation}\label{eq:P-Q-equation}
\begin{aligned}
P_{\ell,m}(b,\tau_0)&=\alpha-m^2+2c+b\cos(m\tau_0)-2c\cos(\vartheta_\ell)\cos(m\tau_c),\\
Q_{\ell,m}(b,\tau_0)&=\eta m-b\sin(m\tau_0)+2c\cos(\vartheta_\ell)\sin(m\tau_c).
\end{aligned}
\end{equation}
Therefore, \(im\) is a characteristic exponent associated with the spatial mode \(\ell\) if and only if
\begin{equation}\label{eq:P-Q-are0}
P_{\ell,m}(b,\tau_0)=0,\qquad Q_{\ell,m}(b,\tau_0)=0.
\end{equation}

For \(m\neq0\), the critical parameter pairs \((b,\tau_0)\) are determined by the simultaneous system
\[
\begin{cases}
\alpha-m^2+2c+b\cos(m\tau_0)-2c\cos(\vartheta_\ell)\cos(m\tau_c)=0,\\
\eta m-b\sin(m\tau_0)+2c\cos(\vartheta_\ell)\sin(m\tau_c)=0.
\end{cases}
\]
For each fixed pair \((\ell,m)\), we define
\[
\mathcal T_{\ell,m}:=\left\{(b,\tau_0)\in\mathfrak P:P_{\ell,m}(b,\tau_0)=0,\quad Q_{\ell,m}(b,\tau_0)=0\right\}.
\]
Under the usual nondegeneracy assumptions, \(\mathcal T_{\ell,m}\) is locally discrete.

Equivalently, whenever \(\sin(m\tau_0)\neq0\), the second equation gives
\[
b=\frac{\eta m+2c\cos(\vartheta_\ell)\sin(m\tau_c)}{\sin(m\tau_0)}.
\]
Substitution into the first equation yields the scalar compatibility condition
\begin{equation}\label{eq:solve-for-tau0}
\alpha-m^2+2c-2c\cos(\vartheta_\ell)\cos(m\tau_c)+\bigl(\eta m+2c\cos(\vartheta_\ell)\sin(m\tau_c)\bigr)\cot(m\tau_0)=0.
\end{equation}
Thus, the admissible delays are the solutions of \eqref{eq:solve-for-tau0}. We denote these solutions by
\[
\tau_{0,\ell,m}^{(j)},\qquad j=1,2,\ldots.
\]
For each such delay, the corresponding critical value of \(b\) is given by
\begin{equation}\label{eq:solve-for-b}
b_{\ell,m}^{(j)}=\frac{
\eta m+2c\cos(\vartheta_\ell)\sin(m\tau_c)}{\sin\bigl(m\tau_{0,\ell,m}^{(j)}\bigr)
}.
\end{equation}
Accordingly, the critical set associated with the spatial mode \(\ell\) and
the temporal mode \(m\) can be written as
\[
\mathcal T_{\ell,m}
=\left\{
p_{\ell,m}^{(j)}
:=\bigl(b_{\ell,m}^{(j)},
\tau_{0,\ell,m}^{(j)}\bigr)
:j=1,2,\ldots
\right\}.
\]

We set
\[
\mathcal T:=\bigcup_{\ell\in\Lambda_n}\bigcup_{m=1}^{\infty}
\mathcal T_{\ell,m}.
\]
Hence, a parameter pair \((b,\tau_0)\) can be a local bifurcation point from the trivial branch only if
\[
(b,\tau_0)\in\mathcal T.
\]

Let
\[
p_*=(b_*,\tau_{0,*})\in\mathcal T
\]
be an isolated critical point. The local bifurcation invariant at \(p_*\) is denoted by
\[
\omega[p_*,0]\in A_1(G).
\]
 It is computed from the linearized operator \(D_x\Psi_p(0)\). Let \(\Sigma_-\) denote the negative spectrum of the finite-dimensional part of the linearized operator, and let \(\mathfrak m_{\ell,m}(\mu)\) be the multiplicity of \(\mathcal V_{\ell,m}\) in the eigenspace corresponding to \(\mu\in\Sigma_-\). We also define the isotypical crossing number
\[
\mathfrak t_{\ell,m}(b_*,\tau_{0,*})
\]
which records the signed crossing of the block \(\mathcal V_{\ell,m}\) at \(p_*\). Only blocks satisfying
\[
\Delta_\ell(im,b_*,\tau_{0,*})=0
\]
can contribute to this term.

For each irreducible isotypical component \(\mathcal V_{\ell,m}\), we define its basic degree by
\[
\deg_{\mathcal V_{\ell,m}}:=
\eqdeg{G}\left(-\id, B_\epsilon(\mathcal V_{\ell,m})\right),
\]
where \(B_\epsilon(\mathcal V_{\ell,m})\) denotes a sufficiently small ball centered at the origin in \(\mathcal V_{\ell,m}\).

For the zero temporal mode \(m=0\), the \(S^1\)-action is trivial. Consequently, the \(D_n\times\mathbb Z_2\times S^1\)-isotypical decomposition of the corresponding subspace coincides with its \(D_n\times\mathbb Z_2\)-isotypical decomposition. To simplify the notation, we therefore write
\[
\mathcal V_{\ell}:=\mathcal V_{\ell,0} \quad \text{ and } \quad\mathfrak m_\ell(\mu):=\mathfrak m_{\ell,0}(\mu).
\]
The temporal mode \(m=0\) corresponds to constant solutions, whereas the modes \(m\geq1\) correspond to nonconstant \(2\pi\)-periodic solutions. Accordingly, whenever the existence of nonconstant periodic solutions is asserted below, we identify the contribution to the bifurcation invariant arising from the modes \(m\geq1\) and consider the maximal finite orbit types associated with those modes.

We denote the set of critical mode indices at \(p_*\) by
\[
\mathcal K:=
\left\{
(\ell,m)\in\Lambda_n\times\mathbb N_+:
\Delta_\ell(im,b_*,\tau_{0,*})=0
\right\}.
\]
By the product formula for the equivariant degree \cite{AED,Yu3},
\begin{equation}\label{degree}
\omega[p_*,0]=
\prod_{\mu \in \Sigma_{-}} \prod_{\ell\in\Lambda_n}
\left(\deg_{\mathcal{V}_{\ell}}\right)^{\mathfrak m_{\ell}(\mu)}
\cdot
\sum_{(\ell,m)\in\mathcal K}
\mathfrak t_{\ell,m}(p_*)
\deg_{\mathcal V_{\ell,m}}
\in A_1(G).
\end{equation}
\begin{theorem}[Local bifurcation criterion]
\label{thm:new-local-bifurcation}
Let
\[
p_*=(b_*,\tau_{0,*})\in\mathcal T
\]
be an isolated critical parameter point. If
\[
\omega[p_*,0]\neq 0
\qquad\text{in }A_1(G),
\]
then \((p_*,0)\) is a bifurcation point of the trivial branch.

Moreover, if \((H)\) is maximal among the orbit types occurring with nonzero coefficient in the contribution to \(\omega[p_*,0]\) from a temporal
mode \(m>0\), then nonconstant \(2\pi\)-periodic solutions bifurcate from \((p_*,0)\) with isotropy containing a subgroup conjugate to \(H\).
\end{theorem}

\begin{proof}
Assume, to the contrary, that no bifurcation occurs at \((p_*,0)\). Then there exist regular parameter values \(p_-\) and \(p_+\), lying on opposite sides of \(p_*\) along a prescribed transversal segment, and a sufficiently small
\(G\)-invariant ball
\[
\mathcal B_\epsilon\subset\mathcal X
\]
such that \(x=0\) is the only zero of \(\Psi_p\) in \(\mathcal B_\epsilon\) for every
\[
p\in[p_-,p_+].
\]
The family
\[
s\longmapsto \Psi_{(1-s)p_-+sp_+},
\qquad 0\leq s\leq1,
\]
is therefore an admissible \(G\)-equivariant homotopy. By homotopy invariance of
the equivariant degree,
\[
\eqdeg{G}\bigl(\Psi_{p_-},\mathcal B_\epsilon\bigr)
=
\eqdeg{G}\bigl(\Psi_{p_+},\mathcal B_\epsilon\bigr).
\]
Hence
\[
\omega[p_*,0]=0,
\]
contradicting the assumption
\[
\omega[p_*,0]\neq0. 
\]
Therefore, \((p_*,0)\) is a bifurcation point.

If \(m>0\), the corresponding temporal mode is nonconstant. The existence
property of the \(G\)-equivariant degree then yields nonconstant periodic solutions near \((p_*,0)\) whose isotropy contains a subgroup conjugate to \(H\).
\end{proof}

\begin{remark}
Since \(\omega[p_*,0]\) lies in \(A_1(G)\), it contains more information than the ordinary topological degree. In particular, it can distinguish spatial and spatio-temporal symmetry types of the bifurcating periodic solutions.
\end{remark}

\section{Global Continuation of Bifurcating Branches}
We now extend the local bifurcation result to a global continuation statement. Since the bifurcation parameter \(p=(b,\tau_0)\)
belongs to the space \(\mathfrak P\), we restrict the problem to a one-dimensional transversal curve through an isolated critical parameter point.

Let \(J\subset\mathbb R\)
be an open interval and let
\[
p_*=(b_*,\tau_{0,*})\in\mathcal T
\]
be isolated, and let
\[
\sigma\in C^1(J,\mathfrak P),
\qquad
\sigma(s)=\bigl(b(s),\tau_0(s)\bigr),
\]
be an admissible crossing curve satisfying \(
\sigma(s_*)=p_*\) for some \(s_*\in J\), and define the restricted family \(\Phi:J\times\mathcal X\longrightarrow\mathcal X\) by
\[
\Phi(s,u):=\Psi_{\sigma(s)}(u).
\]
Since \(G\) acts trivially on the parameter \(s\), the map \(\Phi(s,\cdot)\) is \(G\)-equivariant for every \(s\in J\).

Let
\[
\mathcal S_\sigma
=
\overline{
\left\{
(s,u)\in J\times\mathcal X:
\Phi(s,u)=0,\ u\neq0
\right\}
}.
\]
We denote by
\[
\mathcal C_\sigma(s_*,0)
\]
the connected component of \(\mathcal S_\sigma\) containing \((s_*,0)\). We then obtain the following theorem as a consequence of the Rabinowitz alternative~\cite{Rabinowitz1971}.

\begin{theorem}[Global bifurcation along a transversal]
\label{thm:global-continuation}
Let
\[
p_*=(b_*,\tau_{0,*})\in\mathcal T
\]
be an isolated critical parameter point. Assume that \(s_*\) is an isolated critical value of the restricted family
\[
\Phi(s,u)=\Psi_{\sigma(s)}(u)
\]
and that
\[
\omega[\sigma(s_*),0]\neq0 \qquad\text{in }A_1(G),
\]

Then \(\mathcal C_\sigma(s_*,0)\) contains nontrivial \(2\pi\)-periodic solutions. Moreover, at least one of the following alternatives occurs:
\begin{enumerate}
\item
\(\mathcal C_\sigma(s_*,0)\) is unbounded in \(J\times\mathcal X\);

\item
\(\mathcal C_\sigma(s_*,0)\) meets another critical point
\[
(\widehat s,0),
\quad
\widehat s\neq s_*,
\]
of the trivial branch;

\item
the closure of \(\mathcal C_\sigma(s_*,0)\) meets \(
\partial J\times\mathcal X.\)
\end{enumerate}
\end{theorem}

\begin{proof}
Since \(\omega[\sigma(s_*),0]\neq0,\)
the local bifurcation criterion implies that \((s_*,0)\) is a bifurcation point.
Hence \(\mathcal C_\sigma(s_*,0)\) contains nontrivial periodic solutions.

Assume, to the contrary, that none of the three alternatives occurs. Then \(\mathcal C_\sigma(s_*,0)\) is bounded in \(J\times\mathcal X\), its closure does not intersect \(\partial J\times\mathcal X\), and it does not meet any other critical point of the trivial branch.

Consequently, there exists a bounded open \(G\)-invariant set
\[
\mathcal W\subset J\times\mathcal X
\]
such that
\[
\mathcal C_\sigma(s_*,0)\subset\mathcal W,
\qquad
\overline{\mathcal W}\subset J\times\mathcal X,
\]
and
\[
\Phi(s,u)\neq0
\qquad
\text{for all }(s,u)\in\partial\mathcal W.
\]
We may choose \(\mathcal W\) so that \((s_*,0)\) is the only critical point of the trivial branch contained in \(\mathcal W\).

Choose \(s_*^-<s_*<s_*^+\)
such that \([s_*^-,s_*^+]\subset J\)
and \(s_*\) is the only critical value in this interval. Set
\[
K_*=[s_*^-,s_*^+]\times\{0\}\subset J\times\mathcal X.
\]
Equip \(J\times\mathcal X\) with the product norm
\[
\|(s,u)\|_{J\times\mathcal X}
=
|s|+\|u\|_{\mathcal X},
\]
and let
\[
d_*(s,u)
=
\operatorname{dist}\bigl((s,u),K_*\bigr).
\]
For \(\delta>0\) sufficiently small, define
\[
\mathcal U_*
=
\left\{
(s,u)\in\mathcal W:
d_*(s,u)<\delta
\right\}.
\]

Define the complementing function \(
\vartheta:\mathcal W\longrightarrow\mathbb R\)
by
\[
\vartheta(s,u)=
\begin{cases}
d_*(s,u)-\dfrac{\delta}{2},
&(s,u)\in\mathcal U_*,
\\[2mm]
\dfrac{\delta}{2},
&(s,u)\in\mathcal W\setminus\mathcal U_*.
\end{cases}
\]
The function \(d_*\) is continuous and \(G\)-invariant because \(G\) acts trivially on \(s\), fixes the trivial branch, and acts isometrically on \(\mathcal X\). Moreover, if
\[
(s,u)\in\partial\mathcal U_*\cap\mathcal W,
\]
then
\[
d_*(s,u)=\delta.
\]
Hence the two formulas defining \(\vartheta\) agree on
\(\partial\mathcal U_*\cap\mathcal W\), and therefore \(\vartheta\) is continuous and \(G\)-invariant.

Define the augmented map \(
\mathcal P:\mathcal W\longrightarrow\mathbb R\times\mathcal X\)
by
\[
\mathcal P(s,u)=
\bigl(\vartheta(s,u),\Phi(s,u)\bigr).
\]
 The map \(\mathcal P\) is continuous and \(G\)-equivariant. Furthermore, since
\[
\Phi(s,u)\neq0
\qquad
\text{on }\partial\mathcal W,
\]
the map \(\mathcal P\) has no zeros on \(\partial\mathcal W\). Thus, \(
\eqdeg{G}(\mathcal P,\mathcal W)\) is well defined.

If \(\mathcal P(s,u)=0,\) then \(\Phi(s,u)=0\)
and \(
d_*(s,u)=\frac{\delta}{2}.\)
Therefore all zeros of \(\mathcal P\) lie in \(\mathcal U_*\). By excision of the equivariant degree~\cite{AED},
\[
\eqdeg{G}(\mathcal P,\mathcal W)
=
\eqdeg{G}(\mathcal P,\mathcal U_*).
\]
By the complementing-function formula for the equivariant degree, the latter degree equals the local bifurcation invariant:
\[
\eqdeg{G}(\mathcal P,\mathcal W)
=
\omega[\sigma(s_*),0].
\]

Now define the homotopy
\[
\mathcal H:
[0,1]\times\mathcal W
\longrightarrow
\mathbb R\times\mathcal X
\]
by
\[
\mathcal H(\lambda,s,u)
=
\bigl(
(1-\lambda)\vartheta(s,u)-\lambda,
\Phi(s,u)
\bigr).
\]
Since
\[
\Phi(s,u)\neq0
\qquad
\text{for all }(s,u)\in\partial\mathcal W,
\]
the homotopy \(\mathcal H\) has no zeros on \(\partial\mathcal W\) and is therefore admissible. Homotopy invariance gives
\[
\eqdeg{G}(\mathcal P,\mathcal W)
=
\eqdeg{G}\bigl(\mathcal H(1,\cdot),\mathcal W\bigr).
\]
At \(\lambda=1\),
\[
\mathcal H(1,s,u)
=
\bigl(-1,\Phi(s,u)\bigr),
\]
whose first component is never zero. Hence
\[
\eqdeg{G}\bigl(\mathcal H(1,\cdot),\mathcal W\bigr)=0.
\]
It follows that
\[
\omega[\sigma(s_*),0]=
\eqdeg{G}(\mathcal P,\mathcal W)=0,
\]
contradicting the assumption
\[
\omega[\sigma(s_*),0]\neq0.
\]
Therefore, at least one of the three alternatives must occur.
\end{proof}

\begin{remark}
Theorem~\ref{thm:global-continuation} shows that a branch detected by the nonzero invariant \(
\omega[\sigma(s_*),0]\)
cannot terminate in a bounded neighborhood of \((s_*,0)\). Along the transversal family, the branch must become unbounded, meet another critical point of the trivial branch, or approach an endpoint of \(J\).
\end{remark}

The unbounded alternative can be refined when periodic solutions satisfy an a priori bound on compact subintervals of \(J\).

\begin{condition}\label{cond:new-bounded-on-intervals}
For every compact interval \(
K\subset J,\) 
there exists a constant \(M_K>0\) such that every \(2\pi\)-periodic solution \(u\in\mathcal X\) of
\[
\Phi(s,u)=0,
\qquad
s\in K,
\]
satisfies
\[
\|u\|_{\mathcal X}\leq M_K.
\]
\end{condition}

\begin{corollary}\label{cor:new-parameter-continuation}
Assume that the hypotheses of Theorem~\ref{thm:global-continuation} and Condition~\ref{cond:new-bounded-on-intervals} hold. Suppose further that \(
\mathcal C_\sigma(s_*,0)\)
does not meet another critical point of the trivial branch. Then either \(\operatorname{proj}_{J}\mathcal C_\sigma(s_*,0)\)
is not relatively compact in \(J\), or the closure of
\(\mathcal C_\sigma(s_*,0)\) meets
\[
\partial J\times\mathcal X.
\]
In particular, if \(
J=\mathbb R,\)
then \(\operatorname{proj}_{\mathbb R}\mathcal C_\sigma(s_*,0)
\) is unbounded.
\end{corollary}

\begin{proof}
Suppose that \(
\operatorname{proj}_{J}\mathcal C_\sigma(s_*,0)\)
is contained in a compact interval \(K\subset J\). By
Condition~\ref{cond:new-bounded-on-intervals}, there exists \(M_K>0\) such that
\[
\|u\|_{\mathcal X}\leq M_K
\]
for every
\[
(s,u)\in\mathcal C_\sigma(s_*,0).
\]
Hence \(\mathcal C_\sigma(s_*,0)\) is bounded in \(J\times\mathcal X\).

Since the component does not meet another critical point of the trivial branch, Theorem~\ref{thm:global-continuation} implies that its closure must meet
\[
\partial J\times\mathcal X.
\]
Therefore, if the boundary alternative does not occur, the projection of the component cannot be relatively compact in \(J\).

If \(J=\mathbb R\), then \(\partial J=\varnothing\), so the boundary alternative is absent. Consequently, \(
\operatorname{proj}_{\mathbb R}\mathcal C_\sigma(s_*,0)
\) must be unbounded.
\end{proof}
\section{Crossing Number for Purely Imaginary Characteristic Roots}

In this section, we compute the crossing number associated with the purely imaginary characteristic roots
\[
\nu=\pm i m_*.
\]
We take \(b\) as the bifurcation parameter and keep
\[
\eta,\qquad \alpha,\qquad c,\qquad \tau_{0,*},\qquad \tau_c
\]
fixed. Let
\[
\zeta(b)=\varrho(b)+i\varphi(b)
\]
be a characteristic root such that
\[
\varrho(b_*)=0,
\qquad
\varphi(b_*)=m_*.
\]

Substituting
\[
\zeta(b)=\varrho(b)+i\varphi(b)
\]
into the characteristic equation, we obtain
\[
\Delta_\ell(\varrho+i\varphi,b,\tau_{0,*})
=
R_\ell(\varrho,\varphi,b)
+iI_\ell(\varrho,\varphi,b),
\]
where
\[
\begin{aligned}
R_\ell(\varrho,\varphi,b)
={}&
\varrho^2-\varphi^2+\eta\varrho+\alpha+2c \\
&+b e^{-\varrho\tau_{0,*}}\cos(\varphi\tau_{0,*})
-2c\cos(\vartheta_\ell)e^{-\varrho\tau_c}
\cos(\varphi\tau_c),
\end{aligned}
\]
and
\[
\begin{aligned}
I_\ell(\varrho,\varphi,b)
={}&
2\varrho\varphi+\eta\varphi
-b e^{-\varrho\tau_{0,*}}\sin(\varphi\tau_{0,*}) \\
&+2c\cos(\vartheta_\ell)e^{-\varrho\tau_c}
\sin(\varphi\tau_c).
\end{aligned}
\]
Therefore, the characteristic equation is equivalent to
\[
R_\ell(\varrho,\varphi,b)=0,
\qquad
I_\ell(\varrho,\varphi,b)=0.
\]

Differentiating these equations with respect to \(b\) and evaluating them at
\[
(\varrho,\varphi,b)=(0,m_*,b_*),
\]
gives
\[
\left\{
\begin{aligned}
A_{\ell,m_*}\varrho'+B_{\ell,m_*}\varphi'
&=-\cos(m_*\tau_{0,*}),\\
-B_{\ell,m_*}\varrho'+A_{\ell,m_*}\varphi'
&=\sin(m_*\tau_{0,*}),
\end{aligned}
\right.
\]
where
\[
A_{\ell,m_*}
=
\eta-b_*\tau_{0,*}\cos(m_*\tau_{0,*})
+2c\tau_c\cos(\vartheta_\ell)\cos(m_*\tau_c),
\]
and
\[
B_{\ell,m_*}
=
-2m_*
-b_*\tau_{0,*}\sin(m_*\tau_{0,*})
+2c\tau_c\cos(\vartheta_\ell)\sin(m_*\tau_c).
\]
Here,
\[
\varrho'
=
\left.\frac{d\varrho}{db}\right|_{b=b_*},
\qquad
\varphi'
=
\left.\frac{d\varphi}{db}\right|_{b=b_*}.
\]

If
\[
A_{\ell,m_*}^2+B_{\ell,m_*}^2\neq0,
\]
then
\[
\partial_\nu\Delta_\ell(im_*,b_*,\tau_{0,*})\neq0,
\]
so the critical characteristic root is simple. Solving the preceding system yields
\[
\varrho'
=
-\frac{
A_{\ell,m_*}\cos(m_*\tau_{0,*})
+B_{\ell,m_*}\sin(m_*\tau_{0,*})
}{
A_{\ell,m_*}^2+B_{\ell,m_*}^2
},
\]
and
\[
\varphi'
=
\frac{
A_{\ell,m_*}\sin(m_*\tau_{0,*})
-B_{\ell,m_*}\cos(m_*\tau_{0,*})
}{
A_{\ell,m_*}^2+B_{\ell,m_*}^2
}.
\]
Consequently,
\[
\left.
\frac{d}{db}\operatorname{Re}\zeta(b)
\right|_{b=b_*}
=
-\frac{
A_{\ell,m_*}\cos(m_*\tau_{0,*})
+B_{\ell,m_*}\sin(m_*\tau_{0,*})
}{
A_{\ell,m_*}^2+B_{\ell,m_*}^2
}.
\]

Since
\[
A_{\ell,m_*}^2+B_{\ell,m_*}^2>0,
\]
the sign of the crossing speed is determined by
\begin{equation}\label{eq:sign-equation}
\operatorname{sign}
\left(
\left.
\frac{d}{db}\operatorname{Re}\zeta(b)
\right|_{b=b_*}
\right)
=
-\operatorname{sign}
\left(
A_{\ell,m_*}\cos(m_*\tau_{0,*})
+B_{\ell,m_*}\sin(m_*\tau_{0,*})
\right).
\end{equation}

Therefore, assume that the transversality condition
\[
A_{\ell,m_*}\cos(m_*\tau_{0,*})
+B_{\ell,m_*}\sin(m_*\tau_{0,*})
\neq0
\]
holds. Then the conjugate pair crosses the imaginary axis with nonzero speed as \(b\) passes through \(b_*\). Counting a conjugate pair as one crossing, we define the crossing number by
\begin{equation}\label{eq:new-crossing-number}
\mathfrak t_{\ell,m_*}(p_*)
=
\begin{cases}
-d_{\ell,m_*},&
A_{\ell,m_*}\cos(m_*\tau_{0,*})
+B_{\ell,m_*}\sin(m_*\tau_{0,*})<0,\\[1mm]
d_{\ell,m_*},&
A_{\ell,m_*}\cos(m_*\tau_{0,*})
+B_{\ell,m_*}\sin(m_*\tau_{0,*})>0,
\end{cases}
\end{equation}
where \(d_{\ell,m_*}\) denotes the number of copies of the irreducible
representation \(\mathcal V_{\ell,m_*}\) contained in the critical
eigenspace corresponding to the pair \(\nu=\pm im_*\).

 \section{Examples}
\subsection{Example 1: The Case \(D_7\)}

We illustrate the preceding results for a ring supply-chain network consisting of seven identical warehouses. Let \(n=7\) and
\[
G=D_7\times\mathbb Z_2\times S^1.
\]
The spatial wave numbers are
\[
\vartheta_\ell=\frac{2\pi\ell}{7},\qquad \ell=0,1,\ldots,6.
\]
Since \(n=7\) is odd, the nontrivial spatial Fourier modes occur in the conjugate pairs
\[
\ell=1,6,\qquad \ell=2,5,\qquad \ell=3,4.
\]

We fix
\[
\eta=2,\qquad \alpha=1,\qquad c=1,\qquad \tau_c=\frac{3\pi}{4},
\]
and consider the spatio-temporal block
\[
(\ell_*,m_*)=(1,1).
\]
Thus,
\[
\vartheta_1=\frac{2\pi}{7},
\qquad
m_*=1.
\]

The real and imaginary parts of the characteristic equation at \(\nu=i\) give
\[
b_*\cos(\tau_{0,*})
=1-\alpha-2c+
2c\cos\left(\frac{2\pi}{7}\right)
\cos\left(\frac{3\pi}{4}\right)
\]
and
\[
b_*\sin(\tau_{0,*})=\eta+
2c\cos\left(\frac{2\pi}{7}\right)
\sin\left(\frac{3\pi}{4}\right).
\]
Therefore, using the notation in
\eqref{eq:P-Q-equation}--\eqref{eq:solve-for-b}, we obtain

\[
b_*\approx4.07541\quad \tau_{0,*}\approx2.35619.
\]

Thus,
\[
p_*=(b_*,\tau_{0,*})
\approx(4.07541,2.35619)
\]
is a critical parameter point associated with the block \(\mathcal V_{1,1}\).

We next determine the negative part of the zero temporal mode. The eigenvalue of the linearized fixed-point operator on the block \(\mathcal V_{\ell,m}\) is
\[
\mu_{\ell,m}(b,\tau_0)
=
\frac{
m^2-\alpha-2c
+2c\cos(\vartheta_\ell)e^{-im\tau_c}
-b e^{-im\tau_0}
-i\eta m
}{m^2+1}.
\]
For \(m=0\), the exponential delay factors are equal to \(1\), and therefore
\[
\mu_{\ell,0}(b_*,\tau_{0,*})
=-\alpha-b_*+
2c\left(
\cos\left(\frac{2\pi\ell}{7}\right)-1
\right).
\]

For \(r=1,2,3\), define
\[
a_r
=
\left(
\cos(0\vartheta_r),
\cos(\vartheta_r),
\ldots,
\cos(6\vartheta_r)
\right)^T
\]
and
\[
b_r
=
\left(
\sin(0\vartheta_r),
\sin(\vartheta_r),
\ldots,
\sin(6\vartheta_r)
\right)^T.
\]
The real spatial eigenspaces are
\[
W_0
=
\operatorname{span}
\left\{
(1,1,1,1,1,1,1)^T
\right\},
\]
and
\[
W_r=\operatorname{span}\{a_r,b_r\},
\qquad r=1,2,3.
\]
The zero-frequency eigenvalues are summarized in
Table~\ref{tab:D7-zero-frequency}.
\begin{table}[H]
\caption{Eigenvalues and eigenspaces of the zero temporal mode for Example~1.}
\label{tab:D7-zero-frequency}
\centering
\[
\begin{array}{|c|c|c|c|}
\hline
\text{spatial mode}
&
\ell
&
\mu_{\ell,0}(b_*,\tau_{0,*})
&
\text{eigenspace}
\\
\hline
\text{homogeneous mode}
&
0
&
-5.07541
&
W_0=\operatorname{span}\{(1,1,1,1,1,1,1)^T\}
\\
\hline
\text{first spatial pair}
&
1,6
&
-5.82843
&
W_1=\operatorname{span}\{a_1,b_1\}
\\
\hline
\text{second spatial pair}
&
2,5
&
-7.52045
&
W_2=\operatorname{span}\{a_2,b_2\}
\\
\hline
\text{third spatial pair}
&
3,4
&
-8.87734
&
W_3=\operatorname{span}\{a_3,b_3\}
\\
\hline
\end{array}
\]
\end{table}

Consequently, every component of the zero temporal mode belongs to the negative spectral space.
To identify the relevant irreducible components, we use the character table of \(D_7\times\mathbb Z_2\) in Table~\ref{character-table-d7}:

  \begin{table}[H]
		\centering
		\caption{Character Table of \(D_7\times \bz_2\)}
          \label{character-table-d7}
		\begin{tabular}{@{}lrrrrrrrrrr@{}}
			\toprule
			& \((1)\) & \((-1)\) &  \((\kappa)\) &\((-\kappa)\) & \((r)\) &\((-r)\) & \((r^2)\) & \((-r^2)\) & \((r^3)\) & \((-r^3)\) \\
			\midrule
			\(\mathcal{V}_1\) & \(1\) & \(1\) & \(1\) & \(1\) & \(1\) & \(1\)& \(1\) & \(1\) & \(1\) & \(1\) \\
        \(\mathcal{V}_2\) & \(1\) & \(-1\) & \(-1\) & \(1\) & \(1\) & \(-1\)& \(1\) & \(-1\) & \(1\) & \(-1\)   \\
        			\(\mathcal{V}_3\) & \(1\) & \(-1\) & \(1\) & \(-1\) & \(1\) & \(-1\)& \(1\) & \(-1\) & \(1\) & \(-1\)  \\
        \(\mathcal{V}_4\) & \(1\) & \(1\) & \(-1\) & \(-1\) & \(1\) & \(1\)& \(1\) & \(1\) & \(1\) & \(1\)   \\
           \(\mathcal{V}_5\) & \(2\) & \(-2\) & \(0\) & \(0\) & \(A\) & \(-A\) & \(C\)& \(-C\) & \(B\) & \(-B\)  \\
               \(\mathcal{V}_6\) & \(2\) & \(-2\) & \(0\) & \(0\) & \(B\) & \(-B\) & \(A\)& \(-A\) & \(C\) & \(-C\)  \\
                  \(\mathcal{V}_7\) & \(2\) & \(-2\) & \(0\) & \(0\) & \(C\) & \(-C\) & \(B\)& \(-B\) & \(A\) & \(-A\)  \\
                  \(\mathcal{V}_8\) & \(2\) & \(2\) & \(0\) & \(0\) & \(A\) & \(A\) & \(C\)& \(C\) & \(B\) & \(B\)  \\
               \(\mathcal{V}_9\) & \(2\) & \(2\) & \(0\) & \(0\) & \(B\) & \(B\) & \(A\)& \(A\) & \(C\) & \(C\)  \\
                  \(\mathcal{V}_{10}\) & \(2\) & \(2\) & \(0\) & \(0\) & \(C\) & \(C\) & \(B\)& \(B\) & \(A\) & \(A\)  \\
			\bottomrule
		\end{tabular}
	\end{table}
where \(A=2\cos(\frac{6\pi}{7}), B=2\cos(\frac{4\pi}{7}), C=2\cos(\frac{2\pi}{7})\), we 
have the \(D_7\times\mathbb Z_2\)-isotypical decomposition
\[
\mathbb R^7=\mathcal V_3\oplus\mathcal V_7\oplus\mathcal V_6\oplus
\mathcal V_5,
\]
the components \(\mathcal V_3, \mathcal V_7, \mathcal V_6, \mathcal V_5\) correspond respectively to the real spatial modes \(\ell=0, \ell=1, \ell=2, \ell=3. \)

We now compute the crossing direction of the critical block in the \(b\)-direction while keeping \(
\tau_0=\tau_{0,*} \)
fixed. By~\eqref{eq:sign-equation} and~\eqref{eq:new-crossing-number}, we have 
\[
\left.
\frac{d}{db}\operatorname{Re}\zeta(b)
\right|_{b=b_*}
\approx0.10534>0.
\]
Thus, the conjugate pair \(\nu=\pm i\)
crosses the imaginary axis from left to right as \(b\) increases through \(b_*\). Since the critical eigenspace contains one copy of
\(\mathcal V_{1,1}\), we have \(d_{1,1}=1\), and hence
\[
\mathfrak t_{1,1}(p_*)=-1.
\]
		
By computation in GAP, we have 
		\begin{align*}
		&\eqdeg{G}_{\mathcal V_3}\cdot\eqdeg{G}_{\mathcal V_7}\cdot\eqdeg{G}_{\mathcal V_6}\cdot\eqdeg{G}_{\mathcal V_5}=(D_7^p)-(D_7)-(D_1^z)+(D_1)\\
        &\eqdeg{G}_{\mathcal V_{1,1}}=(\amal{\bz_7^p}{\bz_{14}}{}{\bz_1}{\bz_1})+(\amal{D_1^p}{\bz_2}{}{D_1}{\bz_1})+(\amal{D_1^p}{\bz_2}{}{D_1^z}{\bz_1})-(\amal{\bz_1^p}{\bz_2}{}{\bz_1}{\bz_1})
		\end{align*}
      The group notation used in this paper can be found in \cite{AED,diss}.
We compute the local bifurcation invariant at \((p_*, 0)\):
\[
	\begin{aligned}
 \omega[p_*, 0]&= \eqdeg{G}_{\mathcal V_3}\cdot\eqdeg{G}_{\mathcal V_7}\cdot\eqdeg{G}_{\mathcal V_6}\cdot\eqdeg{G}_{\mathcal V_5}\cdot(-\eqdeg{G}_{\mathcal V_{1,1}})\\
 &=(\amal{\bz_7^p}{\bz_{14}}{}{\bz_1}{\bz_1})-(\amal{\bz_7}{\bz_{7}}{}{\bz_1}{\bz_1})+(\amal{D_1^p}{\bz_{2}}{}{D_1}{\bz_1})+(\amal{D_1^p}{\bz_{2}}{}{D_1^z}{\bz_1})\\
 &\quad -3(\amal{\bz_1^p}{\bz_{2}}{}{\bz_1}{\bz_1})+(\bz_1\times\bz_1).
	\end{aligned}
\]

From the computation of the local bifurcation invariant, we conclude that global bifurcation occurs. More precisely, there exists at least one global branch of nonconstant \(2\pi\)-periodic solutions bifurcating from the trivial branch, and the bifurcating solutions may have one of the following orbit types:
\[
(\amal{\bz_7^p}{\bz_{14}}{}{\bz_1}{\bz_1}),\qquad
(\amal{D_1^p}{\bz_{2}}{}{D_1}{\bz_1}),\qquad
(\amal{D_1^p}{\bz_{2}}{}{D_1^z}{\bz_1}).
\]
We use a Fourier collocation method on \([0,2\pi]\) with \(N=48\) equally
spaced nodes, so that \(h=2\pi/N=\pi/24\). The delayed terms are evaluated
in Fourier space, and the nonlinear periodic branch is continued numerically
using Broyden's method. A representative periodic solution is shown in
Figure~\ref{fig:D7-periodic-solution}, while the corresponding bifurcation
diagram is displayed in Figure~\ref{fig:D7-bifurcation}.

\begin{figure}[H]
    \centering
        \caption{Numerically computed periodic solution for Example 1.}
    \includegraphics[width=0.6\textwidth]{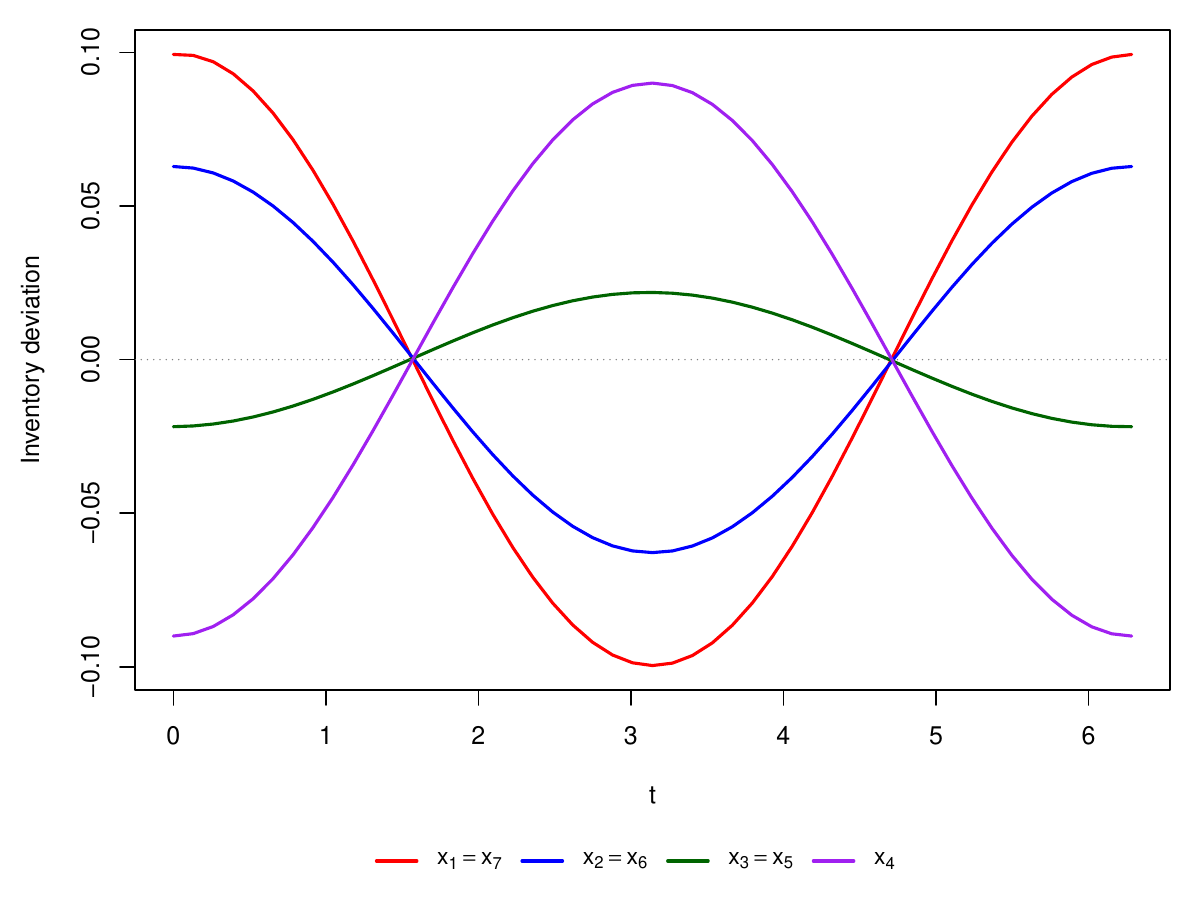}
    \label{fig:D7-periodic-solution}
\end{figure}

\begin{figure}[H]
    \centering
        \caption{Numerical bifurcation diagram for Example~1.}
    \includegraphics[width=0.6\textwidth]{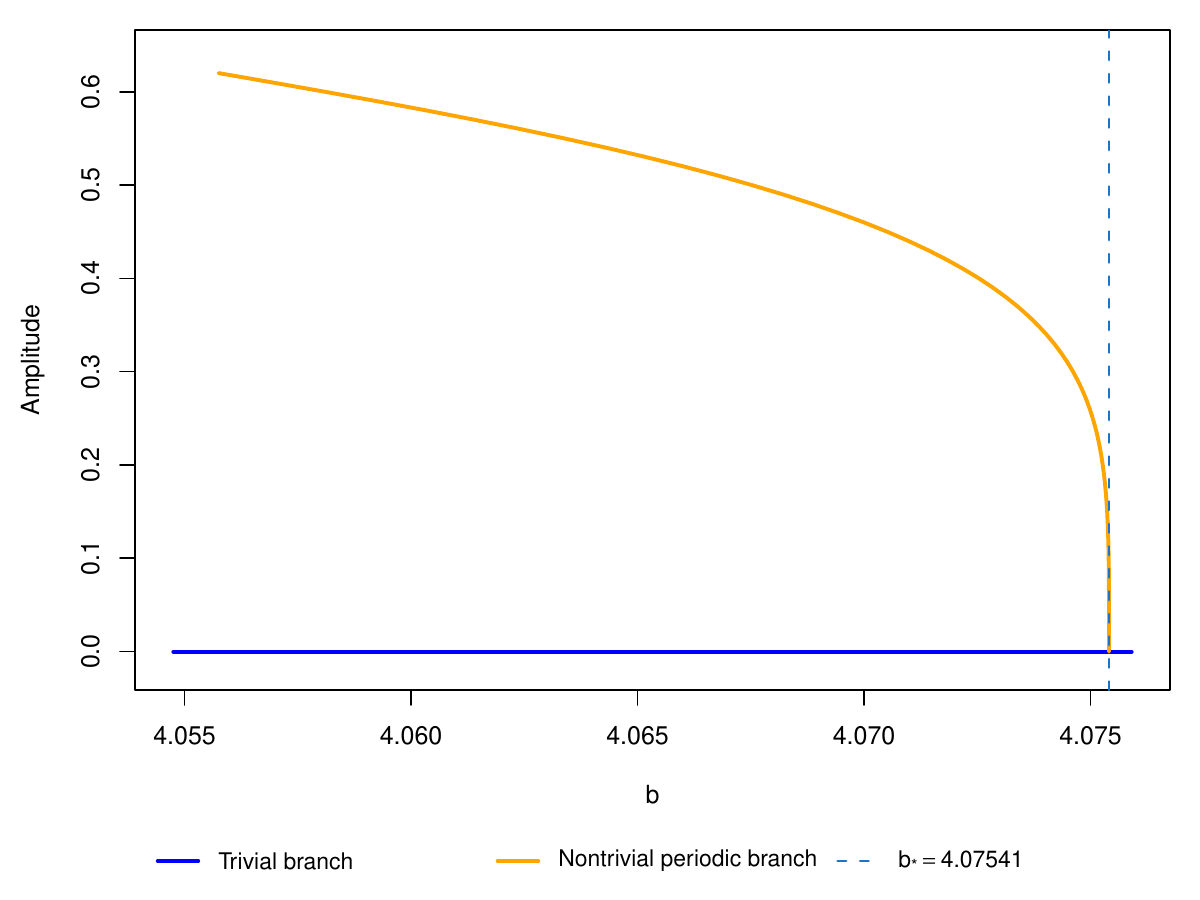}
    \label{fig:D7-bifurcation}
\end{figure}

\subsection{Example 2: The Case \(D_8\)}

We now give a second example for a supply-chain ring with eight warehouses.
Let \(n=8\) and
\[
G=D_8\times\mathbb Z_2\times S^1.
\]
The spatial wave numbers are
\[
\vartheta_\ell=\frac{2\pi\ell}{8},
\qquad
\ell=0,1,\ldots,7.
\]
For \(D_8\), the nontrivial spatial modes occur in the pairs
\[
\ell=1,7,
\qquad
\ell=2,6,
\qquad
\ell=3,5,
\]
together with the special mode
\[
\ell=4.
\]

We fix
\[
\eta=2,
\qquad
\alpha=1,
\qquad
c=0.6,
\qquad
\tau_c=\pi,
\]
and consider the spatio-temporal block
\[
(\ell_*,m_*)=(2,1).
\]
By a similar computation, we obtain the critical parameter point
\[
p_*=(b_*,\tau_{0,*})=
\left(
\frac{2\sqrt{34}}{5},
\pi-\arctan\frac{5}{3}
\right),
\]
associated with the block \(\mathcal V_{2,1}\).

 For \(m=0\), the eigenvalue of the linearized fixed-point operator is
\[
\mu_{\ell,0}(b_*,\tau_{0,*})
=-\alpha-b_*+2c\left(\cos\frac{2\pi\ell}{8}-1
\right).
\]
Define the spatial eigenspaces similarly:
\[
a_r=\left(\cos(0\vartheta_r),\cos(1\vartheta_r),\ldots,\cos(7\vartheta_r)\right)^T,
\]
and
\[
b_r=\left(\sin(0\vartheta_r),\sin(1\vartheta_r),\ldots,\sin(7\vartheta_r)\right)^T,
\]
where
\[
\vartheta_r=\frac{2\pi r}{8},\qquad r=1,2,3.
\]
Then
\[
W_0=\operatorname{span}\{(1,1,1,1,1,1,1,1)^T\},
\]
and, for \(r=1,2,3\),
\[
W_r=\operatorname{span}\{a_r,b_r\}.
\]
The special mode \(\ell=4\) has the one-dimensional eigenspace
\[
W_4=\operatorname{span}\{(1,-1,1,-1,1,-1,1,-1)^T\}.
\]

The \(m=0\) eigenvalues and corresponding spatial eigenspaces are summarized in Table~\ref{tab:D8-zero-frequency}.
\begin{table}[H]
\caption{Eigenvalues and eigenspaces of the zero temporal mode for Example~2.}
\label{tab:D8-zero-frequency}
\centering
\[
\begin{array}{|c|c|c|c|}
\hline
\text{spatial mode} & \ell & \mu_{\ell,0}(b_*,\tau_{0,*}) & \text{eigenspace} \\
\hline
\text{homogeneous mode} & 0 & -3.33238 & W_0=\operatorname{span}\{(1,1,1,1,1,1,1,1)^T\} \\
\hline
\text{first spatial pair} & 1,7 & -3.68385 & W_1=\operatorname{span}\{a_1,b_1\} \\
\hline
\text{second spatial pair} & 2,6 & -4.53238 & W_2=\operatorname{span}\{a_2,b_2\} \\
\hline
\text{third spatial pair} & 3,5 & -5.38091 & W_3=\operatorname{span}\{a_3,b_3\} \\
\hline
\text{alternating mode} & 4 & -5.73238 & W_4=\operatorname{span}\{(1,-1,1,-1,1,-1,1,-1)^T\}\\
\hline
\end{array}
\]
\end{table}
We next compute the crossing number
\[
\operatorname{sign}
\left(
\left.
\frac{d}{db}\operatorname{Re}\zeta(b)
\right|_{b=b_*}
\right)
=
-\operatorname{sign}
\left(
A_{2,1}\cos(\tau_{0,*})
+
B_{2,1}\sin(\tau_{0,*})
\right)
=1.
\]
Thus the conjugate pair
\[
\nu=\pm i
\]
crosses from the left half-plane to the right half-plane as
\(b\) increases through \(b_*\).

Since the critical eigenspace contains one copy of
\(\mathcal V_{2,1}\), we have \(d_{2,1}=1\), and hence
\[
\mathfrak t_{2,1}(p_*)=-1.
\]

To continue, the \(D_8\times\bz_2\) character table is in Table~\ref{table:D_8 Z2 isotypical}:
\footnotesize
  \begin{table}[H]
		\centering
		\caption{Character Table of \(D_8\times \bz_2\)}
          \label{table:D_8 Z2 isotypical}
		\begin{tabular}{@{}lrrrrrrrrrrrrrrrr@{}}
			\toprule
			& \((1)\) & \((-1)\) &  \((\kappa)\) &\((-\kappa)\) & \((\kappa r)\) & \((-\kappa r)\) & \((r)\) &\((-r)\) & \((r^2)\) & \((-r^2)\) & \((r^3)\) & \((-r^3)\) & \((r^4)\) & \((-r^4)\)\\
			\midrule
			\(\mathcal{V}_1\) & \(1\) & \(1\) & \(1\) & \(1\) & \(1\) & \(1\)& \(1\) & \(1\) & \(1\) & \(1\) & \(1\) & \(1\) & \(1\) & \(1\)  \\
        \(\mathcal{V}_2\) & \(1\) & \(-1\) & \(1\) & \(-1\) & \(1\) & \(-1\)& \(1\) & \(-1\) & \(1\) & \(-1\)  & \(1\) & \(-1\) & \(1\) & \(-1\)  \\
        			\(\mathcal{V}_3\) & \(1\) & \(1\) & \(-1\) & \(-1\) & \(-1\) & \(-1\)& \(1\) & \(1\) & \(1\) & \(1\) & \(1\) & \(1\) & \(1\) & \(1\) \\
      \(\mathcal{V}_4\) & \(1\) & \(-1\) & \(-1\) & \(1\) & \(-1\) & \(1\)& \(1\) & \(-1\) & \(1\) & \(-1\) & \(1\) & \(-1\) & \(1\) & \(-1\) \\
       \(\mathcal{V}_5\) & \(1\) & \(1\) & \(1\) & \(1\) & \(-1\) & \(-1\)& \(-1\) & \(-1\) & \(1\) & \(1\) & \(-1\) & \(-1\) & \(1\) & \(1\) \\
          \(\mathcal{V}_6\) & \(1\) & \(-1\) & \(1\) & \(-1\) & \(-1\) & \(1\)& \(-1\) & \(1\) & \(1\) & \(-1\) & \(-1\) & \(1\) & \(1\) & \(-1\) \\
            \(\mathcal{V}_7\) & \(1\) & \(1\) & \(-1\) & \(-1\) & \(1\) & \(1\)& \(-1\) & \(-1\) & \(1\) & \(1\) & \(-1\) & \(-1\) & \(1\) & \(1\) \\
                 \(\mathcal{V}_8\) & \(1\) & \(-1\) & \(-1\) & \(1\) & \(1\) & \(-1\)& \(-1\) & \(1\) & \(1\) & \(-1\) & \(-1\) & \(1\) & \(1\) & \(-1\) \\
           \(\mathcal{V}_9\) & \(2\) & \(2\) & \(0\) & \(0\) & \(0\) & \(0\) & \(0\) & \(0\) & \(-2\) & \(-2\) & \(0\) & \(0\) & \(2\) & \(2\)  \\
    \(\mathcal{V}_{10}\) & \(2\) & \(-2\) & \(0\) & \(0\) & \(0\) & \(0\) & \(0\) & \(0\) & \(-2\) & \(2\) & \(0\) & \(0\) & \(2\) & \(-2\)  \\
    \(\mathcal{V}_{11}\) & \(2\) & \(2\) & \(0\) & \(0\) & \(0\) & \(0\) & \(\sqrt{2}\) & \(\sqrt{2}\) & \(0\) & \(0\) & \(-\sqrt{2}\) & \(-\sqrt{2}\) & \(-2\) & \(-2\)  \\
      \(\mathcal{V}_{12}\) & \(2\) & \(-2\) & \(0\) & \(0\) & \(0\) & \(0\) & \(\sqrt{2}\) & \(-\sqrt{2}\) & \(0\) & \(0\) & \(-\sqrt{2}\) & \(\sqrt{2}\) & \(-2\) & \(2\)  \\
    \(\mathcal{V}_{13}\) & \(2\) & \(2\) & \(0\) & \(0\) & \(0\) & \(0\) & \(-\sqrt{2}\) & \(-\sqrt{2}\) & \(0\) & \(0\) & \(\sqrt{2}\) & \(\sqrt{2}\) & \(-2\) & \(-2\)  \\
        \(\mathcal{V}_{14}\) & \(2\) & \(-2\) & \(0\) & \(0\) & \(0\) & \(0\) & \(-\sqrt{2}\) & \(\sqrt{2}\) & \(0\) & \(0\) & \(\sqrt{2}\) & \(-\sqrt{2}\) & \(-2\) & \(2\)  \\
			\bottomrule
		\end{tabular}
          \normalsize
	\end{table}
    \normalsize
    Then we have \(\mathbb R^8=\mathcal V_2\oplus\mathcal V_{12}\oplus \mathcal V_{10}\oplus \mathcal V_{14}\oplus \mathcal V_{6} \), where \(\mathcal V_2, \mathcal V_{12}, \mathcal V_{10}, \mathcal V_{14},  \mathcal V_{6}\) are associated with eigenvalues when \(\ell=0,1,2,3,4\) respectively, the components \(\mathcal V_2, \mathcal V_{12}, \mathcal V_{10}, \mathcal V_{14}, \mathcal V_6\) correspond respectively to the real spatial modes \(\ell=0, \ell=1, \ell=2, \ell=3, \ell=4.\)

Using the following basic-degree computations,
		\begin{align*}
		\eqdeg{G}_{\mathcal V_{2}}\cdot\eqdeg{G}_{\mathcal V_{12}}\cdot\eqdeg{G}_{\mathcal V_{10}}\cdot\eqdeg{G}_{\mathcal V_{14}}\cdot\eqdeg{G}_{\mathcal V_6}&=(D_8^p)-(D_8^d)-(D_8)-(D_4^d)-(\tilde D_4^d)\\
        &+(D_4)+(\tilde D_2^z)+(\tilde D_2)+(\bz_4^d)-(\bz_2)
		\end{align*}
        \begin{align*}
        \eqdeg{G}_{\mathcal V_{2,1}}=(\amal{D_4^p}{\bz_{2}}{}{D_4^d}{\bz_1})+(\amal{\bz_8^p}{\bz_4}{}{\bz_4^d}{\bz_1})+(\amal{\tilde D_4^p}{\bz_2}{}{\tilde D_4^d}{\bz_1})-(\amal{\bz_4^p}{\bz_2}{}{\bz_4^d}{\bz_1})
		\end{align*}
We compute the following local bifurcation invariant at \((p_*, 0)\):
\[
	\begin{aligned}
 \omega[p_*,0]&= 	\eqdeg{G}_{\mathcal V_{2}}\cdot\eqdeg{G}_{\mathcal V_{12}}\cdot\eqdeg{G}_{\mathcal V_{10}}\cdot\eqdeg{G}_{\mathcal V_{14}}\cdot\eqdeg{G}_{\mathcal V_6}\cdot(-\eqdeg{G}_{\mathcal V_{2,1}})\\
 &=(\amal{D_4^p}{\bz_{2}}{}{D_4^d}{\bz_1})-(\amal{\bz_8^p}{\bz_{4}}{}{\bz_4^d}{\bz_1})-(\amal{\tilde D_4^p}{\bz_{2}}{}{\tilde D_4^d}{\bz_1})+3(\amal{\bz_4^p}{\bz_{2}}{}{\bz_4^d}{\bz_1})\\
 &\quad +(\tilde D_4^d\times\bz_1)+(\amal{\tilde D_4^d}{\bz_{2}}{}{\bz_4^d}{\bz_1})+(\amal{D_4^z}{\bz_{2}}{}{D_2^z}{\bz_1})-2(\amal{\tilde D_4^z}{\bz_{2}}{}{\tilde D_2^z}{\bz_1})+(\amal{\tilde D_4}{\bz_{2}}{}{\tilde D_2}{\bz_1})\\
 &\quad+(\amal{\bz_8^d}{\bz_{4}}{}{\bz_2}{\bz_1})-2(\bz_4^d\times\bz_1)-(D_2^z\times\bz_1)-(\amal{D_2^z}{\bz_{2}}{}{\bz_2}{\bz_1})-(\amal{\bz_4}{\bz_{2}}{}{\bz_2}{\bz_1})\\
 &\quad -(\amal{D_1^p}{\bz_{2}}{}{D_1^z}{\bz_1})-(\amal{D_1^p}{\bz_{2}}{}{D_1}{\bz_1})+(D_1\times\bz_1)+(\amal{D_1}{\bz_{2}}{}{\bz_1}{\bz_1})
	\end{aligned}
\]

From the computation of the local bifurcation invariant, we conclude that global bifurcation occurs. More precisely, there exists at least one global branch of nonconstant \(2\pi\)-periodic solutions bifurcating from the trivial branch. Moreover, the bifurcating solutions may have one of the following orbit types:
\[
(\amal{D_4^p}{\bz_{2}}{}{D_4^d}{\bz_1}),\qquad
(\amal{\bz_8^p}{\bz_{4}}{}{\bz_4^d}{\bz_1}),\qquad
(\amal{\tilde D_4^p}{\bz_{2}}{}{\tilde D_4^d}{\bz_1}).
\]
We omit the corresponding numerical figures for this example, since they are qualitatively similar to those presented for Example~1.

The examples predict nonconstant, wave-like inventory patterns in which the warehouses do not oscillate synchronously but instead experience phase-shifted fluctuations around the ring. The computed crossing direction indicates that increasing the delayed inventory-feedback strength may destabilize the homogeneous equilibrium, provided that the remaining characteristic roots remain stable. The nonzero bifurcation invariant then confirms the emergence of nearby periodic inventory oscillations and restricts their possible symmetry types. From a supply-chain perspective, the result suggests that sufficiently strong delayed feedback can produce persistent inventory cycles that propagate through the network rather than occurring uniformly at all warehouses.

\section{Floquet Multipliers and Leading-Order Stability}
In this section, we relate the bifurcation analysis to the Floquet multipliers of the periodic supply-chain oscillations. The equivariant degree detects the existence and symmetry of bifurcating periodic solutions, whereas Floquet theory describes their linear stability under small perturbations. 
\subsection{Leading variational equation along a small-amplitude branch}

Let \(u_*(t)\) be a nontrivial \(2\pi\)-periodic solution bifurcating from an isolated critical point
\[
p_*=(b_*,\tau_{0,*}).
\]
Writing
\[
u(t)=u_*(t)+w(t)
\]
and retaining only terms linear in \(w\), the variational equation is
\begin{equation}\label{eq:new-variational-equation}
\begin{aligned}
w_k''(t)&+\eta w_k'(t)+\alpha w_k(t)+bw_k(t-\tau_{0})\\
&\quad-c\bigl(w_{k+1}(t-\tau_c)-2w_k(t)+w_{k-1}(t-\tau_c)\bigr)
+3\beta u_{*,k}^{\,2}(t)w_k(t)=0.
\end{aligned}
\end{equation}
Because \(u_*(t)\) is \(2\pi\)-periodic, the coefficient \(u_*^2(t)\) is also \(2\pi\)-periodic, and hence \eqref{eq:new-variational-equation} is a linear delay equation with periodic coefficients.

Since the original system is autonomous, \(u_*'(t)\) solves the variational equation. Therefore \(\rho=1\) is always a Floquet multiplier of a nonstationary periodic orbit. This is the neutral multiplier generated by time translation. Linear orbital stability requires this trivial multiplier to be simple and all remaining Floquet multipliers to lie inside the unit disk. If a nontrivial multiplier satisfies \(|\rho|>1\), then the periodic solution is linearly unstable; see
\cite{HaleVerduynLunel1993,YakubovichStarzhinskii1975}.

Assume that bifurcation occurs from a critical pair \((\ell_*,m_*)\) at
\[
p_*=(b_*,\tau_{0,*})
=\bigl(b_{\ell_*,m_*}^{(j)},\tau_{0,\ell_*,m_*}^{(j)}\bigr)
\in\mathcal T_{\ell_*,m_*}.
\]
For a small-amplitude branch, we write
\begin{equation}\label{eq:small-branch-new}
u_*(t)
=\varepsilon q_{\ell_*}\cos(m_*t)+O(\varepsilon^2),
\qquad 0<|\varepsilon|\ll1,
\end{equation}
where \(q_{\ell_*}\) is a normalized vector in the critical spatial eigenspace.

Because the bifurcation parameters are \((b,\tau_0)\), the parameter displacement along a local branch is described by a curve through \(p_*\). We therefore write
\[
b(\varepsilon)
=b_*+b_2\varepsilon^2+O(\varepsilon^3),
\qquad
\tau_0(\varepsilon)
=\tau_{0,*}+\tau_2\varepsilon^2+O(\varepsilon^3),
\]
where \(b_2\) and \(\tau_2\) are determined by the Lyapunov--Schmidt reduced equation for the chosen local branch.

Using \eqref{eq:small-branch-new}, we obtain
\[
u_*^2(t)
=\varepsilon^2q_{\ell_*}^{\,2}\cos^2(m_*t)+O(\varepsilon^3)
=\frac{\varepsilon^2}{2}q_{\ell_*}^{\,2}
\bigl(1+\cos(2m_*t)\bigr)+O(\varepsilon^3),
\]
where \(q_{\ell_*}^{\,2}\) denotes the componentwise square of \(q_{\ell_*}\).
Hence
\[
3\beta u_*^2(t)w(t)
=
\varepsilon^2\frac{3\beta}{2}q_{\ell_*}^{\,2}w(t)
+\varepsilon^2\frac{3\beta}{2}q_{\ell_*}^{\,2}
\cos(2m_*t)w(t)
+O(\varepsilon^3).
\]
Let \(q_\ell\) be a normalized vector in the \(\ell\)-th spatial eigenspace
and consider perturbations of the form
\[
w(t)=q_\ell z(t).
\]
We restrict attention to perturbation directions that are preserved by the leading cubic correction. More precisely, we assume that
\[
q_{\ell_*}^{\,2}q_\ell=C_\ell q_\ell.
\]
 Since \(q_\ell\) is normalized,
\[
C_\ell
=
\left\langle
q_{\ell_*}^{\,2}q_\ell,q_\ell
\right\rangle.
\]
Under these assumptions, the cubic correction preserves the perturbation
direction \(q_\ell\) to leading order. Consequently, the cubic term
restricted to this mode is
\[
\frac{3\beta\varepsilon^2}{2}
C_\ell
\bigl(1+\cos(2m_*t)\bigr)z(t),
\]
Expanding the delay around \(\tau_{0,*}\), we have
\[
z\bigl(t-\tau_0(\varepsilon)\bigr)
=
z(t-\tau_{0,*})
-\tau_2\varepsilon^2 z'(t-\tau_{0,*})
+O(\varepsilon^3).
\]
Therefore,
\[
\begin{aligned}
b(\varepsilon)
z\bigl(t-\tau_0(\varepsilon)\bigr)
&=
\bigl(b_*+b_2\varepsilon^2\bigr)
\left[
z(t-\tau_{0,*})
-\tau_2\varepsilon^2 z'(t-\tau_{0,*})
\right]
+O(\varepsilon^3)\\
&=
b_*z(t-\tau_{0,*})
+\varepsilon^2 b_2 z(t-\tau_{0,*})
-\varepsilon^2 b_*\tau_2 z'(t-\tau_{0,*})
+O(\varepsilon^3).
\end{aligned}
\]
Hence, relative to the linearized operator at
\((b_*,\tau_{0,*})\), the variation of \(b\) contributes
\[
b_2\varepsilon^2 z(t-\tau_{0,*}),
\]
whereas the variation of \(\tau_0\) contributes
\[
-b_*\tau_2\varepsilon^2 z'(t-\tau_{0,*}).
\]
Therefore the projected variational equation takes the form
\[
\begin{aligned}
\Delta_\ell\left(\frac{d}{dt},b_*,\tau_{0,*}\right)z
&+\varepsilon^2
\left[
\frac{3\beta}{2}C_\ell
+\frac{3\beta}{2}C_\ell\cos(2m_*t)
\right]z\\
&+\varepsilon^2 b_2 z(t-\tau_{0,*})
-\varepsilon^2 b_*\tau_2 z'(t-\tau_{0,*})
+O(\varepsilon^3)=0,
\end{aligned}
\]
where
\[
\Delta_\ell\left(\frac{d}{dt},b_*,\tau_{0,*}\right)z
=
z''+\eta z'+(\alpha+2c)z
+b_*z(t-\tau_{0,*})
-2c\cos(\vartheta_\ell)z(t-\tau_c).
\]

\subsection{Nonresonant perturbation of characteristic exponents}

Suppose that the unperturbed characteristic equation in the \(\ell\)-th spatial direction has a purely imaginary root
\[
\Delta_\ell(i\mathfrak k,b_*,\tau_{0,*})=0.
\]
We seek a perturbed characteristic exponent in the form
\[
\nu(\varepsilon)=i\mathfrak k+\varepsilon^2\nu_2+O(\varepsilon^3).
\]
The oscillatory factor \(\cos(2m_*t)\) couples \(e^{i\mathfrak kt}\) to \(e^{i(\mathfrak k+2m_*)t}\) and \(e^{i(\mathfrak k-2m_*)t}\). We call the perturbation nonresonant when
\[
\Delta_\ell(i(\mathfrak k+2m_*),b_*,\tau_{0,*})\neq0,\qquad \Delta_\ell(i(\mathfrak k-2m_*),b_*,\tau_{0,*})\neq0.
\]
Under this assumption, the oscillatory correction does not contribute to the leading shift of the exponent.

Expanding the characteristic function at \(\nu=i\mathfrak k\) and including the order-\(\varepsilon^2\) parameter and averaged cubic corrections gives
\[
\partial_\nu\Delta_\ell(i\mathfrak k,b_*,\tau_{0,*})\,\nu_2+b_2e^{-i\mathfrak k\tau_{0,*}}-i\mathfrak kb_*\tau_2e^{-i\mathfrak k\tau_{0,*}}+\frac{3\beta}{2}C_\ell=0.
\]
Thus
\begin{equation}\label{eq:new-nu2}
\nu_2=-\frac{b_2e^{-i\mathfrak k\tau_{0,*}}-i\mathfrak kb_*\tau_2e^{-i\mathfrak k\tau_{0,*}}+\frac{3\beta}{2}C_\ell}{\partial_\nu\Delta_\ell(i\mathfrak k,b_*,\tau_{0,*})}.
\end{equation}

For the present characteristic function,
\[
\partial_\nu\Delta_\ell(\nu,b,\tau_0)=2\nu+\eta-b\tau_0e^{-\nu\tau_0}+2c\cos(\vartheta_\ell)\tau_c e^{-\nu\tau_c}.
\]
Hence
\[
\partial_\nu\Delta_\ell(i\mathfrak k,b_*,\tau_{0,*})=2i\mathfrak k+\eta-b_*\tau_{0,*}e^{-i\mathfrak k\tau_{0,*}}+2c\cos(\vartheta_\ell)\tau_c e^{-i\mathfrak k\tau_c}.
\]

The corresponding Floquet multiplier is
\[
\rho(\varepsilon)=e^{2\pi\nu(\varepsilon)},
\]
and therefore
\[
|\rho(\varepsilon)|=\exp\left(2\pi\varepsilon^2\operatorname{Re}\nu_2+O(\varepsilon^3)\right).
\]
Consequently, \(\operatorname{Re}\nu_2>0\) implies that the corresponding multiplier moves outside the unit circle to leading order, while \(\operatorname{Re}\nu_2<0\) implies that it moves inside the unit circle.

We therefore obtain the following result.

\begin{theorem}[Leading-order Floquet stability test]\label{thm:floquet-stability-test}
Assume that
\[
\Delta_\ell(i\mathfrak k,b_*,\tau_{0,*})=0
\]
for some integer \(\mathfrak k\), and assume that
\[
\Delta_\ell(i(\mathfrak k+2m_*),b_*,\tau_{0,*})\neq0,\qquad \Delta_\ell(i(\mathfrak k-2m_*),b_*,\tau_{0,*})\neq0,
\]
and
\[
\partial_\nu\Delta_\ell(i\mathfrak k,b_*,\tau_{0,*})\neq0.
\]
Then the corresponding characteristic exponent has the expansion
\[
\nu(\varepsilon)=i\mathfrak k+\varepsilon^2\nu_2+O(\varepsilon^3),
\]
where
\[
\nu_2=-\frac{b_2e^{-i\mathfrak k\tau_{0,*}}-i\mathfrak kb_*\tau_2e^{-i\mathfrak k\tau_{0,*}}+\frac{3\beta}{2}C_\ell}{2i\mathfrak k+\eta-b_*\tau_{0,*}e^{-i\mathfrak k\tau_{0,*}}+2c\cos(\vartheta_\ell)\tau_c e^{-i\mathfrak k\tau_c}}.
\]
If \(\operatorname{Re}\nu_2>0\), then the bifurcating periodic solution is linearly unstable in this direction to leading order. If \(\operatorname{Re}\nu_2<0\), then the corresponding nontrivial Floquet multiplier moves inside the unit circle to leading order.
\end{theorem}

\subsection{Resonant mode coupling}

The nonresonance condition fails when
\[
\Delta_\ell(i(\mathfrak k+2m_*),b_*,\tau_{0,*})=0
\]
or
\[
\Delta_\ell(i(\mathfrak k-2m_*),b_*,\tau_{0,*})=0.
\]
In this case, the oscillatory coefficient with frequency \(2m_*\) couples the mode \(\mathfrak k\) to another critical temporal mode. The correction must then be computed on the corresponding finite-dimensional resonant subspace rather than from the scalar formula \eqref{eq:new-nu2}. The stability criterion derived above is valid under the stated nonresonance condition. If resonance occurs, the coupled critical modes require a separate finite-dimensional reduction, which is beyond the scope of the present analysis.

\section{Conclusion}

This paper developed an equivariant bifurcation framework for periodic inventory oscillations in delayed ring supply-chain networks. The model was reduced to a second-order delay equation and reformulated as an equivariant compact perturbation of the identity, incorporating the spatial symmetry of the ring, sign symmetry, and time translations.

Critical parameter pairs were identified through purely imaginary characteristic roots, and the associated crossing numbers and basic equivariant degrees were combined to construct a local bifurcation invariant. A nonzero invariant guarantees nonconstant periodic solutions and provides information about their possible spatial and spatio-temporal symmetry types. By restricting the two-parameter problem to a transversal curve, a global continuation alternative was also established.

The \(D_7\)- and \(D_8\)-symmetric examples illustrate the computation of critical parameters, crossing directions, and bifurcation invariants. The results predict global branches of nonsynchronous, wave-like inventory oscillations, with numerical continuation used to illustrate the \(D_7\) case.

Finally, Floquet theory was used to study the leading-order stability of small-amplitude periodic branches. Under a nonresonance condition, the behavior of nontrivial Floquet multipliers near the unit circle can be determined from the leading correction to the characteristic exponent.

Future work may extend the Floquet analysis to resonant mode interactions, where several critical temporal modes must be treated simultaneously. It would also be of interest to investigate more general network topologies, heterogeneous warehouses, multiple or distributed delays, and models incorporating time-dependent demand. Such extensions may further clarify how network structure and delayed feedback influence the formation, stability, and propagation of inventory oscillations.

\appendix

\section{Basic Terminology from Equivariant Theory}

We recall several standard notions used throughout the paper. Let \(G\) be a compact Lie group and let \(H\leq G\) be a closed subgroup. The normalizer of \(H\) in \(G\) is denoted by
\[
N(H):=\{g\in G:gHg^{-1}=H\}.
\]
The corresponding Weyl group is
\[
W(H):=N(H)/H.
\]
We write \((H)\) for the conjugacy class of \(H\) in \(G\). The family of all conjugacy classes of closed subgroups of \(G\) is denoted by \(\Phi(G)\).

The set \(\Phi(G)\) is partially ordered by declaring
\[
(H)\leq (K)
\]
whenever there exists \(g\in G\) such that
\[
gHg^{-1}\leq K.
\]
For \(k\geq 0\), we define
\[
\Phi_k(G):=\{(H)\in\Phi(G):\dim W(H)=k\}.
\]
In particular,
\[
\Phi_0(G):=\{(H)\in\Phi(G):\dim W(H)=0\}
\]
is the collection of orbit types with finite Weyl group.

Let \(X\) be a \(G\)-space. For a point \(x\in X\), its isotropy subgroup is
\[
G_x:=\{g\in G:gx=x\}.
\]
The orbit of \(x\) is
\[
G(x):=\{gx:g\in G\}.
\]
Thus \(G(x)\) is naturally identified with the homogeneous space \(G/G_x\). The quotient space formed by all \(G\)-orbits in \(X\), endowed with the quotient topology, is denoted by \(X/G\) and is called the orbit space.

The conjugacy class \((G_x)\) is called the orbit type of \(x\). We write
\[
\Phi(G;X):=\{(H)\in\Phi(G):H=G_x \text{ for some } x\in X\}
\]
for the set of orbit types occurring in \(X\). If \(H\leq G\) is a closed subgroup, we use the notation
\[
X_H:=\{x\in X:G_x=H\}
\]
and
\[
X^H:=\{x\in X:G_x\geq H\}.
\]
The space \(X^H\) is the \(H\)-fixed-point subspace of \(X\).

If \(X\) and \(Y\) are \(G\)-spaces, a continuous map \(f:X\to Y\) is called \(G\)-equivariant if
\[
f(gx)=gf(x)
\]
for all \(g\in G\) and \(x\in X\). When the action of \(G\) on \(Y\) is trivial, the same condition reduces to
\[
f(gx)=f(x),
\]
and in this case \(f\) is called \(G\)-invariant.

Finally, if \(L\leq H\leq G\), define
\[
N(L,H):=\{g\in G:gLg^{-1}\leq H\}.
\]
For \((L),(H)\in\Phi_0(G)\), the number
\[
n(L,H):=\left|N(L,H)/N(H)\right|
\]
is finite and well defined. Equivalently, \(n(L,H)\) counts the number of subgroups \(H'\leq G\) that are conjugate to \(H\) and contain \(L\). Further details and properties of this number can be found in \cite{AED}.

\section{The Burnside Ring \(A_0(G)\)}

Let \(G\) be a compact Lie group. The Burnside ring of \(G\) is the free abelian group generated by the conjugacy classes in \(\Phi_0(G)\). More precisely, we set
\[
A_0(G):=\mathbb Z[\Phi_0(G)].
\]
Hence an element \(a\in A_0(G)\) has the form
\[
a=n_1(H_1)+\cdots+n_m(H_m),
\]
where \(n_i\in\mathbb Z\) and \((H_i)\in\Phi_0(G)\).

The additive structure of \(A_0(G)\) is the usual formal addition of such finite sums. The multiplication is defined by orbit-type decomposition. Namely, for \((H),(K)\in\Phi_0(G)\), one sets
\[
(H)\cdot (K)=\sum_{(L)\in\Phi_0(G)}n_L(L),
\]
where \(n_L\) is the number of \(G\)-orbits of type \((L)\) in the product space \(G/H\times G/K\), equipped with the diagonal \(G\)-action. With this product, \(A_0(G)\) becomes a ring whose identity element is \((G)\).

The partial order on \(\Phi_0(G)\) makes it possible to compute the coefficients \(n_L\) recursively. The multiplication coefficients are given by
\[
n_L=\frac{n(L,H)|W(H)|n(L,K)|W(K)|-\sum_{(\widetilde L)>(L)}n(L,\widetilde L)n_{\widetilde L}|W(\widetilde L)|}{|W(L)|}.
\]
This formula is especially useful in explicit computations of products in \(A_0(G)\).

\section{\(G\)-Equivariant Brouwer Degree}

\label{subsec:G-degree}

Let \(V\) be an orthogonal \(G\)-representation, and let \(f:V\to V\) be a \(G\)-equivariant map. Let \(\Omega\subset V\) be open, bounded, and \(G\)-invariant. The map \(f\) is said to be \(\Omega\)-admissible if
\[
f(x)\neq 0
\]
for every \(x\in \partial\Omega\). In this situation, the pair \((f,\Omega)\) is called an admissible \(G\)-pair. We denote the class of all admissible \(G\)-pairs by \(\mathcal M^G\).

The equivariant Brouwer degree is characterized by a collection of axioms. We state the version needed in this paper; see \cite{AED} for a detailed treatment.

\begin{theorem}\label{thm:GpropDeg}
There exists a unique map
\[
\deg_G:\mathcal M^G\to A_0(G),
\]
called the \(G\)-equivariant Brouwer degree, such that for every admissible \(G\)-pair \((f,\Omega)\),
\[
\deg_G(f,\Omega)=\sum_{(H)\in\Phi_0(G)}n_H(H),
\]
and the following properties hold.

\begin{itemize}
\item[\rm (G1)] \textbf{Existence.}
If
\[
\deg_G(f,\Omega)\neq 0,
\]
then \(f\) has at least one zero in \(\Omega\). More precisely, if the coefficient of \((H)\) is nonzero, then there exists \(x\in\Omega\) such that
\[
f(x)=0
\]
and
\[
(G_x)\geq (H).
\]

\item[\rm (G2)] \textbf{Additivity.}
Let \(\Omega_1\) and \(\Omega_2\) be disjoint open \(G\)-invariant subsets of \(\Omega\). If
\[
f^{-1}(0)\cap\Omega\subset \Omega_1\cup\Omega_2,
\]
then
\[
\deg_G(f,\Omega)=\deg_G(f,\Omega_1)+\deg_G(f,\Omega_2).
\]

\item[\rm (G3)] \textbf{Homotopy invariance.}
If \(h:[0,1]\times V\to V\) is a \(G\)-equivariant homotopy such that \(h_t\) is \(\Omega\)-admissible for every \(t\in[0,1]\), then
\[
\deg_G(h_t,\Omega)
\]
is independent of \(t\).

\item[\rm (G4)] \textbf{Normalization.}
If \(\Omega\) is an open bounded \(G\)-invariant neighborhood of \(0\) in \(V\), then
\[
\deg_G(\id,\Omega)=(G).
\]

\item[\rm (G5)] \textbf{Product property.}
For two admissible \(G\)-pairs \((f_1,\Omega_1)\) and \((f_2,\Omega_2)\),
\[
\deg_G(f_1\times f_2,\Omega_1\times\Omega_2)=\deg_G(f_1,\Omega_1)\cdot \deg_G(f_2,\Omega_2),
\]
where the product on the right-hand side is the product in the Burnside ring \(A_0(G)\).

\item[\rm (G6)] \textbf{Suspension.}
Let \(W\) be another orthogonal \(G\)-representation, and let \(\mathscr B\subset W\) be an open bounded \(G\)-invariant neighborhood of \(0\). Then
\[
\deg_G(f\times \id_W,\Omega\times\mathscr B)=\deg_G(f,\Omega).
\]

\item[\rm (G7)] \textbf{Recurrence formula.}
If
\[
\deg_G(f,\Omega)=\sum_{(H)\in\Phi_0(G)}n_H(H),
\]
then the coefficients \(n_H\) are determined recursively by
\[
n_H=\frac{\deg(f^H,\Omega^H)-\sum_{(K)>(H)}n_K\,n(H,K)\,|W(K)|}{|W(H)|}.
\]
Here \(f^H=f|_{V^H}\), \(\Omega^H=\Omega\cap V^H\), and \(\deg(f^H,\Omega^H)\) denotes the classical Brouwer degree on the fixed-point subspace \(V^H\).
\end{itemize}
\end{theorem}

The same construction also admits an infinite-dimensional version, namely the \(G\)-equivariant Leray--Schauder degree for compact perturbations of the identity in Banach \(G\)-spaces; see \cite{AED}.

\section{Basic Degrees and the Degree of Linear Maps}

Let \(\epsilon>0\) is sufficient small, and define
\[
B_\epsilon(V):=\{x\in V:|x|<\epsilon\}
\]
be the \(\epsilon\)-ball in a finite-dimensional orthogonal \(G\)-representation \(V\). For an irreducible \(G\)-representation \(\mathcal V_i\), the associated basic \(G\)-degree is defined by
\[
\deg_{\mathcal V_i}:=\deg_G(-\id,B_\epsilon(\mathcal V_i)).
\]

Now let \(T:V\to V\) be a \(G\)-equivariant linear isomorphism. Suppose that \(V\) decomposes into \(G\)-isotypical components as
\[
V=\bigoplus_i V_i.
\]
Let \(T_i:=T|_{V_i}\). If \(\mu\in\Sigma(T)\) is a negative eigenvalue, denote its eigenspace by \(E(\mu)\). Let \(m_i(\mu)\) be the multiplicity with which the representation \(\mathcal V_i\) occurs in \(E(\mu)\). Then the product property of the \(G\)-degree gives
\[
\deg_G(T,B_\epsilon(V))=\prod_i \deg_G(T_i,B_\epsilon(V_i)).
\]
Equivalently,
\[
\deg_G(T,B_\epsilon(V))=\prod_i\prod_{\mu\in\Sigma_-(T)}\left(\deg_{\mathcal V_i}\right)^{m_i(\mu)},
\]
where \(\Sigma_-(T)\) denotes the set of negative eigenvalues of \(T\).

The coefficients of each basic degree can be computed from the recurrence formula. Namely, if
\[
\deg_{\mathcal V_i}=\sum_{(H)\in\Phi_0(G)}n_H(H),
\]
then
\[
n_H=\frac{(-1)^{\dim \mathcal V_i^H}-\sum_{(K)>(H)}n_K\,n(H,K)\,|W(K)|}{|W(H)|}.
\]
This expression follows from applying the recurrence formula to the map \(-\id\) on the irreducible representation \(\mathcal V_i\), since the ordinary Brouwer degree of \(-\id\) on the fixed-point space \(\mathcal V_i^H\) is
\[
(-1)^{\dim \mathcal V_i^H}.
\]

\section{Computational Formula for One-Parameter Basic Maps}

We also recall the computational formula used for basic maps arising from one-parameter \(S^1\)-type components. Let \(\mathcal V_{j,1}\) be the representation under consideration, and assume that the relevant twisted orbit types form the finite partially ordered set
\[
\Phi_1(G,\mathcal V_{j,1})=\{(H_1),(H_2),\ldots,(H_m)\}.
\]
Choose a total ordering compatible with the partial order, written as
\[
(H_1)>(H_2)>\cdots>(H_m).
\]

The corresponding basic degree has the form
\[
\deg_{\mathcal V_{j,1}}=\sum_{(H)\in\Phi_1(G)}n_H(H).
\]
The coefficients are obtained recursively. For each \(k=1,\ldots,m\),
\[
n_{H_k}=
\frac{\deg_1(f^{H_k},\Omega^{H_k})-\sum_{l=1}^{k-1}n_{H_l}n(H_k,H_l)|W(H_l)/S^1|}{|W(H_k)/S^1|}.
\]
Here \(\deg_1(f^{H_k},\Omega^{H_k})\) denotes the coefficient corresponding to \((\mathbb Z_1)\) in the \(S^1\)-degree of the restricted map on the fixed-point space. In the present situation this coefficient is given by
\[
\deg_1(f^{H_k},\Omega^{H_k})=\frac{1}{2}\dim \mathcal V_{j,1}^{H_k}.
\]
Therefore, once the values \(n(L,H)\), the Weyl-group orders, and the dimensions of the fixed-point spaces \(\mathcal V_{j,1}^{H_k}\) are known, the twisted basic degree can be computed recursively.

\end{document}